\documentclass[12pt,a4paper]{article}%
\usepackage{amssymb}
\usepackage{amsmath}
\usepackage{amsfonts}
\usepackage[dvips]{graphicx}
\usepackage{mathrsfs}
\usepackage[pdftex]{hyperref}
\usepackage{pdfpages}%
\usepackage[all]{xy}
\providecommand{\U}[1]{\protect\rule{.1in}{.1in}}
\hypersetup{
	colorlinks=true,
	citecolor = blue
}
\newtheorem{theorem}{Theorem}

\newtheorem{definition}[theorem]{Definition}
\newtheorem{example}[theorem]{Example}

\newtheorem{lemma}[theorem]{Lemma}

\newtheorem{proposition}[theorem]{Proposition}
\newtheorem{remark}[theorem]{Remark}

\newenvironment{proof}[1][Proof]{\noindent\textbf{#1.} }{\ \rule{0.5em}{0.5em}}
\begin{document}
	
	\author{V\'{\i}ctor A. Vicente-Ben\'{\i}{}tez\\{\small Departamento de F\'isica, Universidad de Guanajuato, Campus Le\'on}\\{\small Lomas del Bosque 103, Lomas del Campestre, Le\'on, Guanajuato C.P. 37150 M\'{e}xico }\\{\small  va.vicente@ugto.mx, } }
	\title{Transmutation operators and complete systems of solutions for the radial Bicomplex Vekua equation}
	\date{}
	\maketitle

\begin{abstract}
	The construction of a pair of transmutation operators for the  radial main Vekua equation with a Bicomplex-valued coefficient is presented. The pair of operators transforms the Bicomplex analytic functions into the solutions of the main Vekua equation. The analytical properties of the operators in the space of classical solutions and the pseudoanalytic Bergman space are established. With the aid of such operators, a complete system of solutions for the main Vekua equation, called radial formal powers, is obtained. The completeness of the radial formal powers is proven with respect to the uniform convergence on compact subsets and in the Bicomplex pseudoanalytic Bergman space with the $L_2$-norm. For the pseudoanalytic Bergman space on a disk, we show that the radial formal powers are an orthogonal basis, and a representation for the Bergman kernel in terms of such basis is obtained. 
\end{abstract}



\section{Introduction}
In this work, we study the main Vekua equation 
\begin{equation}\label{mainvekuaintro}
	\overline{\boldsymbol{\partial}}W(z)= \frac{\overline{\boldsymbol{\partial}}f(z)}{f(z)}\overline{W}(z), \quad z \in \Omega,
\end{equation}
where $\Omega\subset \mathbb{C}$ is a bounded domain star-shaped with respect to $z=0$, $f$ is a complex valued function that does not vanish in the whole domain $\Omega$ and only depends on the radial component $r=|z|$, and   $\overline{\boldsymbol{\partial}}:= \frac{1}{2}\left(\frac{\partial }{\partial x}+\mathbf{j}\frac{\partial }{\partial y}\right)$ is the Bicomplex Cauchy-Riemann operator. The conjugation $\overline{W}$ is respect to the Bicomplex unit $\mathbf{j}$. Solutions of Eq. (\ref{mainvekuaintro}) are called {\it pseudoanalytic functions} (or {\it generalized analytic functions}) whose theory was introduced in \cite{bers,vekua0} for the complex case. Bicomplex Vekua equations appear, for example, in the study of the Dirac system with fixed energy and with a scalar potential in the two-dimensional case \cite{camposmendez,castanieda}.  Solutions of (\ref{mainvekuaintro}) are closely related to the solutions of the two-dimensional Schr\"odinger equation. Indeed, if $W=u+\mathbf{j}v$ is a Bicomplex solution of (\ref{mainvekuaintro}), where $u$ and $v$ are scalar complex valued functions, then $u$ is a solution of the Schr\"odinger equation $\bigtriangleup u-qu=0$ with $q=\frac{\bigtriangleup f}{f}$, and $v$ is a soltion of the Darboux transformed equation \cite{kravpseudo1,pseudoanalyticvlad}. In \cite{CamposBicomplex}, a generalization of the results of the Bers-Kravchenko theory developed in \cite{bers,pseudoanalyticvlad}, to the case of Bicomplex functions, is presented.

Theory of transmutation operators, also called transformation operators, is a widely used tool in studying partial differential equations, in particular, in the study of  Schr\"odinger and the main Vekua equations \cite{BegehrGilbert,camposmendez,hugo,pablo,transmutationdarboux,transhiperbolic,mine1}. In the case of the radial Schr\"odinger equation, S. Bergman (in \cite{bergman1} for the case of an analytic potential) and R. Gilbert and K. Atkinson (in\cite{gilbertatkinson} for $C^1$-potentials) showed that any classical solution is the image of a harmonic function $h$ under the action of the integral operator
\[
\mathbf{T}_fh(z)= h(z)+\int_0^1\sigma G^f(r,1-\sigma^2)h(\sigma^2 z)d\sigma,
\]
where $G^f$ is a $C^2$ function. In \cite{mine1}, the properties of the operator $\mathbf{T}_f$ are studied, and it is shown that $\mathbf{T}_f$ is a transmutation operator that relates the operators $r^2(\bigtriangleup-q)$ and $r^2\bigtriangleup$. 

The aim of this paper is to construct a pair of transmutation operators that relates the main Vekua operator $\overline{\boldsymbol{\partial}}-\frac{\overline{\boldsymbol{\partial}}f}{f}C$ (where $C$ is the Bicomplex conjugation) with the Bicomplex Cauchy Riemann operator. Following the ideas presented in \cite{camposmendez,transhiperbolic,transmutationdarboux} (where the Vekua equation with a function $f$ that only depends on the variable $x$, is studied), we found a pair of relations between the operator $\mathbf{T}_f$ and their Darboux transformed operator, that is, the operator $\mathbf{T}_{\frac{1}{f}}$ for the associated Darboux transformed radial Schr\"odinger equation. With the aid of such relations, we construct a pair of transmutation operators $\mathcal{T}_f$ and $\mathcal{T}_{\frac{1}{f}}$ whose scalar and vectorial parts are given precisely $\mathbf{T}_f$ and $\mathbf{T}_{\frac{1}{f}}$. As a consequence, a complete system of solutions for (\ref{mainvekuaintro}), that we called the {\it Bicomplex radial formal powers}, is obtained as a result of transmuting the Bicomplex powers $\widehat{z}^n:= (x+\mathbf{j}y)^n$, with $x,y\in \mathbb{R}$, $n\in \mathbb{N}_0$. The completeness of the formal powers in the topology of the uniform convergence on compact subsets of $\Omega$ is shown. In the same way, the properties of the pseudoanalytic Bergman space of $L_2$-solutions of (\ref{mainvekuaintro}) is established. We establish the completeness of the pseudoanalytic Bergman space and the existence of a reproducing Bergman kernel. These results generalize those obtained in \cite{delgadoleblond,camposbergman}, for the case of the complex Vekua equation. We shown that the operators $\mathcal{T}_f$ and $\mathcal{T}_{\frac{1}{f}}$ are bounded in the Bergman space, and we obtain the conditions in the domain $\Omega$ so that the radial formal powers are a complete system in the $L_2$-norm. Furthermore, in the particular case when $\Omega$ is a disk, the radial formal powers are an orthogonal basis.

The paper is structured as follows. Section 2 presents an overview of the main results concerning to the algebra of Bicomplex numbers and Bicomplex holomorphic functions. In section 3, the basic properties of the solutions of the main Vekua equation are presented, together with the main properties of the pseudoanalytic Bergman space, and the corresponding Bergman reproducing kernel. In section 4 we present the construction of a pair of transmutation operators, that transmute the Bicomplex holomorphic functions in solutions of the radial main Vekua equation. The analytical properties of the operators are studied, both in the topology of the uniform convergence on compact subsets, and in the pseudoanalytic Bergman space. Section 5 presents the construction of the radial formal powers. It is shown that this system is complete, both in the topology of the uniform convergence on compact subsets and in the pseudoanalytic Bergman space  with the $L_2$-norm. In the particular case when the domain $\Omega$ is a disk, it is shown that the formal powers are an orthogonal basis for the pseudoanalytic Bergman space. 

\section{Bicomplex analytic functions}
This section summarizes the main properties of Bicomplex numbers and Bicomplex valued analytic functions of one complex variables, which can be found, for example, in \cite{CamposBicomplex,rochon,pseudoanalyticvlad,ShapiroBicomplex,perezregalado}.
\subsection{Basic properties}
Consider the set $\mathbb{C}^2$ with the usual addition and multiplication by an scalar. The product of two elements $W=(w_1,w_2), Z=(z_1,z_2)\in \mathbb{C}^2$ is defined as follows: $WZ:=(w_1z_1-w_2z_2,w_1z_2+w_2z_1)$. With this product, $\mathbb{C}^2$ is a commutative algebra over $\mathbb{C}$ with unitary element $(1,0)$, and it is called the algebra of the {\it Bicomplex  numbers}, that is denoted by $\mathbb{B}$. Denote $\mathbf{j}:=(0,1)$. Identifying $a\in \mathbb{C}$ with $(a,0)$ we have that any $W\in \mathbb{B}$ can be written in a unique form as $W=u+\mathbf{j}v$, with $u,v\in \mathbb{C}$. We call $u$ and $v$ the {\it Scalar} and {\it Vectorial} parts of $W$, respectively, and are denoted by  $\operatorname{Sc}W=u$ and $\operatorname{Vec}W=v$. We have that $\mathbf{j}^2=-1$ and $i\mathbf{j}=\mathbf{j}i$.  The bicomplex conjugate of $W$ is defined by $\overline{W}:= u-\mathbf{j}v$. Note that $W\overline{W}=u^2-v^2\in \mathbb{C}$. The algebra $\mathbb{B}$ contains zero divisors. Indeed, consider 
\[
\mathbf{k}:=i\mathbf{j}, \qquad \mathbf{p}^{\pm} := \frac{1}{2}(1\pm\mathbf{k}),
\]
hence $\mathbf{p}^{+}\mathbf{p}^{-}=0$, $(\mathbf{p}^{\pm})^2=\mathbf{p}^{\pm}$ and $\mathbf{p}^{+}+\mathbf{p}^{-}=1$. The set of the zero divisors of $\mathbb{B}$ is denoted by $\sigma(\mathbb{B})$. The following proposition summarizes some properties of the Bicomplex numbers shown in \cite{CamposBicomplex, ShapiroBicomplex}.
 
\begin{proposition}\label{properties1}
	Let $W,V\in \mathbb{B}$.
\begin{enumerate}
	\item $W\in \sigma(\mathbb{B})$ iff $W\overline{W}=0$.
	\item If $W\overline{W}\neq 0$, $W$ is invertible and $W^{-1}= \frac{\overline{W}}{W\overline{W}}$.
	\item There exist unique $W^{\pm}\in \mathbb{C}$ such that $W=\mathbf{p}^{+}W^{+}+\mathbf{p}^{-}W^{-}$. Furthermore, $W^{\pm}=\operatorname{Sc}W\mp i\operatorname{Vec}W$.
	\item $W\in\sigma(\mathbb{B})$ iff $W=\mathbf{p}^{+}W^{+}$ or $W=\mathbf{p}^{-}W^{-}$.
	\item The product $WV$ can be written as
	\begin{equation}\label{producbicomplex}
		WV= \mathbf{p}^{+}W^{+}V^{+}+\mathbf{p}^{+}W^{-}V^{-}.
	\end{equation}
In particular
\begin{equation}\label{powerbicomplex}
	W^n= \mathbf{p}^{+}(W^{+})^n+\mathbf{p}^{-}(W^{-})^n, \quad \forall n\in \mathbb{N}.
\end{equation}
\end{enumerate}	
\end{proposition}
Denote $\mathcal{R}(\mathbb{B}):= \mathbb{B}\setminus \sigma(\mathbb{B})$. Thus, by Proposition \ref{properties1}(1)-(2), $W\in \mathcal{R}(\mathbb{B})$ iff $W$ is invertible, and from (\ref{producbicomplex}) we have
\begin{equation}\label{inversebicomplex}
	W^{-1}=\frac{\mathbf{p}^{+}}{W^{+}}+\frac{\mathbf{p}^{-}}{W^{-}}.
\end{equation}
From this, formula (\ref{powerbicomplex}) is valid until for $n\in \mathbb{Z}$.  For $W\in \mathbb{B}$ define 
\begin{equation}\label{involution1}
W^{\dagger}:= \mathbf{p}^{+}(W^{+})^{*}+\mathbf{p}^{-}(W^{-})^{*},	
\end{equation}
 where $^{*}$ denotes the standard complex conjugation. Operation $W\mapsto W^{\dagger}$ is an involution in $\mathbb{B}$ \cite[Ch. I]{ShapiroBicomplex}. For $W,V\in \mathbb{B}$ define
\[
\langle W,V \rangle_{\mathbb{B}} := \operatorname{Sc} WV^{\dagger} = \frac{W^{+}(V^{+})^*+ W^{-}(V^{-})^*}{2}.
\]
It is not difficult to see that $\langle W,V \rangle_{\mathbb{B}}$ is a complex inner product in $\mathbb{B}$, and we have the norm
\[
|W|_{\mathbb{B}}:= \sqrt{\langle W,W \rangle_{\mathbb{B}}} = \sqrt{\frac{|W^{+}|^2+|W^{-}|^2}{2}}.
\]
where $|W^{\pm}|$ denotes the complex absolute value. The following inequality is immediate
\begin{equation}\label{equivalentnorms}
	\frac{1}{\sqrt{2}}|W^{\pm}|\leqslant |W|_{\mathbb{B}}\leqslant \frac{1}{\sqrt{2}}\left(|W^+|+|W^{-}|\right).
\end{equation}
\begin{remark}\label{remarknorm} 
A direct computation shows the following relations
\begin{eqnarray}
	W^{\dagger} & = &  \left(\operatorname{Sc}W\right)^*-\mathbf{j}\left(\operatorname{Vec}W\right)^* \label{involutionsecondform} \\
	\langle W,V \rangle_{\mathbb{B}} & = &  (\operatorname{Sc}W)(\operatorname{Sc}V)^*+(\operatorname{Vec}W)(\operatorname{Vec}V)^*, \label{innerproduct2ndfor} \\
	|W|_{\mathbb{B}} & = &  \sqrt{|\operatorname{Sc}W|^2+|\operatorname{Vec}W|^2}, \label{normsecondform}\\
	WV^{\dagger} & = & \langle W,V \rangle_{\mathbb{B}}+\mathbf{j}\langle W,\mathbf{j}V\rangle_{\mathbb{B}} \label{bicomplexinnerproduct} \\
	|WV|_{\mathbb{B}} & \leqslant &  \sqrt{2}|W|_{\mathbb{B}}|V|_{\mathbb{B}}.	\label{norminequality}
\end{eqnarray}
  Hence the Bicomplex product is a continuous operation in $\mathbb{B}$ and $\mathcal{R}(\mathbb{B})$ is open \cite[Prop. 3]{CamposBicomplex}.
\end{remark}

The exponential of a Bicomplex number $W$ is defined as follows \cite{CamposBicomplex}
\begin{equation}\label{expdef}
	e^W:= \mathbf{p}^+e^{W^+}+\mathbf{p}^{-}e^{W^{-}}.
\end{equation}
The Bicomplex exponential satisfies the properties: (i) $e^{W+V}=e^We^V$ for all $W,V\in \mathbb{B}$; (ii) $e^W\in \mathcal{R}(\mathbb{B})$ and $(e^W)^{-1}=e^{-W}$ for all $W\in \mathbb{B}$ (see \cite{CamposBicomplex}).
\subsection{Bicomplex analytic functions}
Let $\Omega\subset \mathbb{C}$ be a bounded domain. Consider a Bicomplex-valued function of the complex variable $z=x+iy\in \Omega$, $F:\Omega \rightarrow \mathbb{B}$. Hence $F=f_1+\mathbf{j}f_2$, where $f_1(z)=\operatorname{Sc}F(z)$ and $f_2(z)=\operatorname{Vec}F(z)$. By Proposition \ref{properties1}(3) we can write $F=\mathbf{p}^{+}F^{+}+\mathbf{p}^{-1}F^{-}$ where $F^{\pm}=f_1\mp if_2$. We say that $F\in C^k(\Omega; \mathbb{B})$ ($k\in \mathbb{N}_0$) iff $f_1,f_2\in C^k(\Omega)$. Similarly, if $1\leqslant p\leqslant \infty$, $F\in L_p(\Omega;\mathbb{B})$ iff $f_1,f_2\in L_p(\Omega; \mathbb{B})$. This is equivalent to the condition $F^{\pm}\in L_p(\Omega)$. For the case $p=2$, $L_2(\Omega; \mathbb{B})$ is a complex Hilbert space with the inner product and the norm
\begin{equation}\label{L2bicomplexnorm}
\langle F,G \rangle_{L_2(\Omega; \mathbb{B})} := \iint_{\Omega}\langle F(z),G(z)\rangle_{\mathbb{B}}dA_z, \quad \|F\|_{L_2(\Omega; \mathbb{B})}= \left(\iint_{\Omega}|F(z)|_{\mathbb{B}}^2dA_z\right)^{\frac{1}{2}}	
\end{equation}
Since $L_2(\Omega; \mathbb{B})$ can be regarded as the product space $L_2(\Omega)\times L_2(\Omega)$, and $L_2(\Omega)$ is separable \cite[Sec. 4.3]{brezis}, then $L_2(\Omega;\mathbb{B})$ is a separable complex Hilbert space. The Sobolev spaces $W^{2,k}(\Omega; \mathbb{B})$, with $k\in \mathbb{N}$, are defined in a similar way. Given $z=x+iy$ with $x,y\in \mathbb{R}$, we denote
\begin{equation}\label{bicomplexification}
\widehat{z}:=x+\mathbf{j}y \in \mathbb{B}.	
\end{equation}
 Note that if $z,z_0\in \mathbb{C}$, then 
\begin{equation}\label{variablez}
	\widehat{z}-\widehat{z_0}= \mathbf{p}^{+}(z^*-z^*_0)+\mathbf{p}^{-}(z^*-z_0^*).
\end{equation}
Denote $\mathbb{N}_0:= \mathbb{N}\cup \{0\}$. From Proposition \ref{properties1}(5) we have
\begin{equation}\label{powerz}
	\left(\widehat{z}-\widehat{z_0}\right)^n = \mathbf{p}^{+}(z^*-z_0^*)^n+\mathbf{p}^{-}(z-z_0)^n, \quad \forall n\in \mathbb{N}_0.
\end{equation}

\begin{remark}\label{remarpropbicomplexification}
	\begin{itemize}
		\item[(i)] A direct computation shows that $\mathbb{C}\ni z\mapsto \widehat{z}\in \mathbb{B}$ is a monomorphism of algebras, then $\{x+jy\,|\, x,y\in \mathbb{R}\}$ is a field isomorphic to $\mathbb{C}$,  and $\widehat{z}\dagger= \overline{\widehat{z}}= \widehat{z^*}$. 
		\item[(ii)] Using (\ref{variablez}) and (\ref{expdef}), it is straightforward to show that $e^{\widehat{z}}=e^x\left(\cos(y)+\mathbf{j}\sin(y)\right)$.
	In particular
	\begin{equation}\label{bicomplexcircle}
		e^{\mathbf{j}\theta}= \cos(\theta)+\mathbf{j}\sin(\theta), \quad \forall \theta \in \mathbb{R}.
	\end{equation}
	\end{itemize}
\end{remark}

The Bicomplex {\it Cauchy-Riemann} operators are defined by
\begin{equation}\label{defCR}
	\boldsymbol{\partial} := \frac{1}{2}\left(\frac{\partial}{\partial x}-\mathbf{j}\frac{\partial}{\partial y}\right), \quad \overline{\boldsymbol{\partial}} :=\frac{1}{2}\left(\frac{\partial}{\partial x}+\mathbf{j}\frac{\partial}{\partial y}\right).
\end{equation} The following relations hold
\begin{equation}\label{CRdecomposition}
	\overline{\boldsymbol{\partial}}= \mathbf{p}^{+}\frac{\partial}{\partial z}+\mathbf{p}^{-}\frac{\partial}{\partial z^*}  , \quad \boldsymbol{\partial} = \mathbf{p}^{+}\frac{\partial}{\partial z^*}+\mathbf{p}^{-}\frac{\partial}{\partial z},
\end{equation}
where $\frac{\partial}{\partial z}:= \frac{1}{2}\left(\frac{\partial}{\partial x}-i\frac{\partial}{\partial y}\right)$ and $\frac{\partial}{\partial z^*}:= \frac{1}{2}\left(\frac{\partial}{\partial x}+i\frac{\partial}{\partial y}\right)$ are the usual complex Cauchy-Riemann operators (see \cite{CamposBicomplex}). A function $W\in C^1(\Omega; \mathbb{B})$ is called $\mathbb{B}$-holomorphic or $\mathbb{B}$-analytic (anti-holomorphic or anti-analytic) if $\overline{\boldsymbol{\partial}}W=0$ ($\boldsymbol{\partial} W=0$). 

\begin{remark}\label{remarkdecompositionholomorphic}
\begin{itemize} 
\item[(i)]From Proposition \ref{properties1}(5) and (\ref{CRdecomposition}), a function  $W\in C^1(\Omega;\mathbb{B})$ is $\mathbb{B}$-holomorphic (anti-holomorphic) iff $W^{\pm}$ are anti-holomorphic and holomorphic, respectively, in the complex sense.

\item[(ii)]By (\ref{CRdecomposition}) and (\ref{powerz}), the powers $\{(\widehat{z}-\widehat{z_0})^n\}_{n=0}^{\infty}$ are $\mathbb{B}$-holomorphic and $\boldsymbol{\partial}(\widehat{z}-\widehat{z_0})^n=n(\widehat{z}-\widehat{z_0})^{n-1}$. In a similar way, the powers $\{(\widehat{z}^{\dagger}-\widehat{z}_0^{\dagger})^n \}_{n=0}^{\infty}$ are $\mathbb{B}$-anti-holomorphic. By (\ref{expdef}) , $e^{\widehat{z}}$ is $\mathbb{B}$-holomorphic.
\end{itemize}
\end{remark}

The space $C(\Omega; \mathbb{B})$ can be endowed with a standard topology as follows. Let $\{K_n\}_{n=1}^{\infty}$ be a sequence of compact subsets of $\Omega$ such that: (I) $K_n \subset \operatorname{Int} K_{n+1}$\footnote{$\operatorname{Int}K_n$ denotes the interior of the set $K_n$}, $\forall n\in \mathbb{N}$; (II) $\Omega=\bigcup_{n=1}^{\infty}K_n$. 

One can choose, for example, $K_n= \left\{z\in \Omega\, |\, \operatorname{dist}(z,\partial\Omega)\geqslant \frac{1}{n} \right\}\cap B_n^{\mathbb{C}}(0)$.  The topology of $C(\Omega;\mathbb{B})$ is generated by the family of semi-norms $\left\{\|F\|_n:= \displaystyle \max_{z\in K_n}|F(z)|_{\mathbb{B}} \right\}_{n=1}^{\infty}$. With this topology $C(\Omega;\mathbb{B})$ is a Fr\'echet space. Properties (I) and (II) implies that for any compact $K\subset \Omega$, there exists $N\in \mathbb{N}$ such that $K\subset K_N$. Then the convergence in $ C(\Omega; \mathbb{B})$ is induced by the uniform convergence on compact subsets of $\Omega$. The topology is independent of the choice of $\{K_n\}_{n=1}^{\infty}$ (for the proof of these facts see
\cite[Ch. VII, Sec. 1]{conway} or \cite[Prop. 4. 39]{folland}).
\begin{remark}\label{remarkconvergence}
Using Remark \ref{remarknorm} and (\ref{equivalentnorms}), it is not difficult to see that a sequence $\{W_n\}\subset C(\Omega; \mathbb{B})$ converges to $W\in C(\Omega;\mathbb{B})$ in this topology iff $\{\operatorname{Sc}W_n\}$ and $\{\operatorname{Vec}W_n\}$  converges uniformly on compact subsets of $\Omega$ to $\operatorname{Sc}W$ and $\operatorname{Vec}W_n$, respectively (or equivalently, if $\{W_n^{\pm}\}$ converges uniformly on compact subsets to $W^{\pm}$).
	
\end{remark}
 
\begin{proposition}\label{propsitiontaylor}
	Denote by $\operatorname{Hol}(\Omega; \mathbb{B})$ the class of $\mathbb{B}$-holomorphic functions in $\Omega$.
	\begin{itemize}
		\item[(i)] $\operatorname{Hol}(\Omega; \mathbb{B})$ is a closed subspace of $C(\Omega; \mathbb{B})$ and $\boldsymbol{\partial}: \operatorname{Hol}(\Omega; \mathbb{B}) \rightarrow \operatorname{Hol}(\Omega; \mathbb{B})$ is a continuous operator.
		\item[(ii)] $W\in  \operatorname{Hol}(\Omega; \mathbb{B})$ iff for all $z_0\in \Omega$, $W$ can be written as its Taylor series in the disk $B_R^{\mathbb{C}}(z_0)\subset \Omega$, with $R=\operatorname{dist}(z_0,\partial\Omega)$, that is, 
		\begin{equation}\label{bicomplextaylorseries}
			W(z)= \sum_{n=0}^{\infty}\frac{\boldsymbol{\partial}^nW(z_0)}{n!}(\hat{z}-\hat{z_0})^n, \quad  z\in B_R^{\mathbb{C}}(z_0),
		\end{equation}
	 and series converges absolutely and uniformly on compact subsets of $B_R^{\mathbb{C}}(z_0)$.
	\end{itemize}
\end{proposition}
\begin{proof}
    \begin{itemize}
        \item[(i)] That  $\operatorname{Hol}(\Omega; \mathbb{B})$ is closed follows from Remark \ref{remarkdecompositionholomorphic}(i) and the fact that complex holomorphic and anti-holomorphic functions are closed in $C(\Omega)$. The continuity of $\boldsymbol{\partial}$ follows from (\ref{CRdecomposition}) and the corresponding continuity of the complex Cauchy-Riemann operators in the spaces of complex holomorphic and anti-holomorphic functions (see \cite[Ch. VII, Sec. 2]{conway}). 
        \item[(ii)] It follows from the fact that $W^+$ and $W^{-}$ are  complex holomorphic and anti-holomorphic, from (\ref{CRdecomposition}) and from Remark \ref{remarkconvergence}.

    \end{itemize}
\end{proof}

\subsection{The Bicomplex analytic Bergman space}\label{secbicomplexanalytic}

We introduce the Bicomplex analytic Bergman space in $\Omega$ as
\begin{equation}
	\mathcal{A}^2(\Omega; \mathbb{B}):= \left\{W\in \operatorname{Hol}(\Omega; \mathbb{B})\, |\, W\in L_2(\Omega; \mathbb{B})  \right\}.
\end{equation}

Let $W\in \mathcal{A}^2(\Omega; \mathbb{B})$. By Remark \ref{remarkdecompositionholomorphic}(i) and the definition of $L_2(\Omega; \mathbb{B})$, $W^+\in \overline{\mathcal{A}}^2(\Omega)$ and $W^{-}\in \mathcal{A}^2(\Omega)$, where $\mathcal{A}^2(\Omega)$ and $\overline{\mathcal{A}}^2(\Omega)$ are  the complex analytic and  anti-analytic Bergman spaces. Then $\mathcal{A}^2(\Omega; \mathbb{B})$ can be regarded as the direct sum of the complex Hilbert spaces $\mathcal{A}^2(\Omega)\oplus\overline{\mathcal{A}}^2(\Omega)$. Thus, $\mathcal{A}^2(\Omega; \mathbb{B})$ is  a complex Hilbert space, and since $\mathcal{A}^2(\Omega;\mathbb{B})$ is subspace of the separable complex Hilbert space $L_2(\Omega;\mathbb{B})$, then $\mathcal{A}^2(\Omega;\mathbb{B})$ is separable \cite[Prop. 3.25]{brezis}. For each $z\in \Omega$, a direct computation shows that $W(z)$ can be written in terms of the Bergman kernels $L_{\Omega}(z,\zeta)$ and $K_{\Omega}(z,\zeta)$ of $\overline{\mathcal{A}}^2(\Omega)$ and $\mathcal{A}^2(\Omega)$ as follows

\begin{equation*}
	W(z) =\iint_{\Omega} W(z)\left(\mathbf{p}^+L_{\Omega}(z,\zeta)+\mathbf{p}^{-}K_{\Omega}(z,\zeta)\right)dA_{\zeta}.
\end{equation*}
Then the function $\mathscr{K}_{\Omega}(z,\zeta):= P^+L_{\Omega}(z,\zeta)+P^{-}K_{\Omega}(z,\zeta)$ has similar properties to a reproducing Bergman kernel. In particular, for $\Omega=\mathbb{D}$, with $\mathbb{D}:= B_1^{\mathbb{C}}(0)$, we obtain 
\begin{equation}\label{Bergmankernel}
	\mathscr{K}_{\mathbb{D}}(z,\zeta)=P^+\frac{1}{1-z\zeta^*}+P^{-}\frac{1}{1-z^*\zeta }= (1-\hat{z}\hat{\zeta}\dagger)^{-1}
\end{equation}

\begin{proposition}\label{proporthonormalbasispowers}
	Consider $\Omega= B_R^{\mathbb{C}}(0)$, for some $R>0$. Then the powers $\{\hat{z}^n, \mathbf{j}\hat{z}^n \}$ are an orthogonal basis for the complex Hilbert space $\mathcal{A}^2(B_R^{\mathbb{C}}(0); \mathbb{B})$.
\end{proposition}
\begin{proof}
	Similar to the complex case \cite[Ch. I]{duren}, using the characterization of the inner product in $\mathbb{B}$ in terms of the scalar and vectorial part, a simple computation with polar coordinates shows the following relations for each $n,m\in \mathbb{N}_0$
	\begin{equation}\label{innerproducpowers}
	 \langle \widehat{z}^n, \mathbf{j}\widehat{z}^m\rangle_{L_2(B_R^{\mathbb{C}}(0); \mathbb{B})}= 0, \quad \langle \Lambda\widehat{z}^n, \Lambda\widehat{z}^m\rangle_{L_2(B_R^{\mathbb{C}}(0); \mathbb{B})}= \frac{2\pi R^{n+m+2}}{n+m+2} \delta_{m,n}, \; \mbox{ where  } \Lambda \in \{1, \mathbf{j}\}.
	\end{equation}
In particular $\left\|\Lambda \hat{z}^n \right\|_{L_2(B_R^{\mathbb{C}}(0); \mathbb{B})}= \sqrt{\frac{\pi}{n+1}}R^{n+1}$, $n\in \mathbb{N}_0$, $\Lambda\in \{1,\mathbf{j}\}$. The proof of completeness is similar to that of the classical complex case given in \cite[Ch. I]{duren}, adapted to the inner product (\ref{L2bicomplexnorm}) and using representation (\ref{bicomplextaylorseries})
\end{proof}

The corresponding orthonormal basis is given by $\{e_n(z; R), \mathbf{j}e_n(z; R)\}_{n=0}^{\infty}$, where 
\begin{equation}\label{orthonormalbasispowers}
	e_n(z; R):= \sqrt{\frac{n+1}{\pi }}\frac{\widehat{z}^n}{R^{n+1}}.
\end{equation}	

\begin{remark}\label{basismodulo}
	The system $\{e_n(z; R), \mathbf{j}e_n(z; R)\}_{n=0}^{\infty}$ is an orthonormal basis for the {\bf complex} Hilbert space $\mathcal{A}^2(B_R^{\mathbb{C}}(0); \mathbb{B})$. If $W\in \mathcal{A}^2(B_R^{\mathbb{C}}(0); \mathbb{B})$, due to (\ref{bicomplexinnerproduct}) we have 
	\begin{equation*}
			W(z) = \sum_{n=0}^{\infty}\left\{\iint_{B_R^{\mathbb{C}}(0)}W(z)\left(e_n(z; R)\right)^{\dagger}dA_z\right\}e_n(z;R).
	\end{equation*}
 In this sense $\{e_n(z; R)\}_{n=0}^{\infty}$ is an orthonormal ``basis" if we consider Bicomplex coefficients. Actually, the spaces $L_2(B_R^{\mathbb{C}}(0); \mathbb{B})$ and $\mathcal{A}^2(B_R^{\mathbb{C}}(0); \mathbb{B})$ are $\mathbb{B}$-modules and it is possible to introduce a $\mathbb{B}$-valued inner product taking the integral $\iint_{B_{\rho}^{\mathbb{C}}(0)} W(z)V^{\dagger}(z)dA_z$ (note that  $\langle W, V\rangle_{L_2(B_R^{\mathbb{C}}(0); \mathbb{B})}$ is the scalar part of this integral). It is possible to develop a theory for $\mathbb{B}$-valued Hilbert spaces, see, e.g., \cite{rochon,ShapiroBicomplex,perezregalado}.	
\end{remark}

\section{The Bicomplex Vekua equation}

\subsection{Main results about the Vekua equation}
Let $\Omega\subset \mathbb{C}$ be a bounded domain and $f\in C^2(\Omega)$ a scalar function that does not vanish in $\Omega$ (i.e., $f(z)\neq 0$ for all $z\in \Omega$). Set $q_f= \frac{\bigtriangleup f}{f}$, where $\bigtriangleup:= \frac{\partial^2}{\partial x^2}+\frac{\partial^2}{\partial y^2}$ is the Laplacian operator. The study of the Schr\"odinger operator 
\begin{equation}\label{schrodingereq}
	\bigtriangleup f(z)-q_f(z)f(z)=0, \quad z\in \Omega,
\end{equation} 
for the case when $f$ is real-valued leads us to the following factorization (see \cite{kravpseudo1})
\begin{align*}
	\frac{1}{4}\left(\bigtriangleup -q_f(z)\right)\varphi(z) &= \left(\frac{\partial}{\partial z^*}+\frac{f_{z}}{f}C\right)\left(\frac{\partial}{\partial z}-\frac{f_{z}}{f}C\right)\varphi(z)=\left(\frac{\partial}{\partial z}+\frac{f_{z^*}}{f}C\right)\left(\frac{\partial}{\partial z^*}-\frac{f_{z^*}}{f}C\right)\varphi(z),
\end{align*}
valid for real-valued functions $\varphi\in C^2(\Omega)$ (here, $f_{z^*}:= \frac{\partial f}{\partial z^*}$ and $C$ denotes the operator of complex conjugation). In the case of a complex-valued function $f$, it is known that the factorization is given in terms of Bicomplex-valued functions and the Bicomplex Cauchy-Riemann operators.

\begin{theorem}[\cite{pseudoanalyticvlad}, Th. 148]
	Let $f\in C^2(\Omega)$ a scalar complex-valued function. Then for any scalar complex-valued function $\varphi\in C^2(\Omega)$, the following equalities hold
	\begin{equation}\label{complexfactorization}
	  \frac{1}{4}\left(\bigtriangleup -q_f(z)\right)\varphi(z) = \left(\overline{\boldsymbol{\partial}}+\frac{\boldsymbol{\partial}f}{f}C_{\mathbb{B}}\right)\left(\boldsymbol{\partial}-\frac{\boldsymbol{\partial}f}{f}C_{\mathbb{B}}\right)\varphi(z)=\left(\boldsymbol{\partial}+\frac{\overline{\boldsymbol{\partial}}f}{f}C_{\mathbb{B}}\right)\left(\overline{\boldsymbol{\partial}}-\frac{\overline{\boldsymbol{\partial}}f}{f}C_{\mathbb{B}}\right)\varphi(z),	
	\end{equation}
where $C_{\mathbb{B}}$ denotes the operator of Bicomplex conjugation.
\end{theorem}

The {\bf main Bicomplex Vekua equation} associated to $f$ is given by
\begin{equation}\label{mainvekua}
	\overline{\boldsymbol{\partial}}W(z)= \frac{\overline{\boldsymbol{\partial}}f(z)}{f(z)}\overline{W}(z), \quad z\in \Omega.
\end{equation}
We consider classical solutions $W\in C^2(\Omega; \mathbb{B})$ of (\ref{mainvekua}). These functions are a special kind of pseudoanalytic functions (see \cite[Ch. II]{pseudoanalyticvlad}). The following result holds.

\begin{theorem}[\cite{pseudoanalyticvlad}, Th. 150]\label{theoremvekuaconjugate}
	Let $W\in C^1(\Omega; \mathbb{B})$. If $W$ is a solution of (\ref{mainvekua}), then $u=\operatorname{Sc}W$ is a solution of (\ref{schrodingereq}), and $v=\operatorname{Vec}W$ is a solution of the Darboux associated equation
	\begin{equation}\label{darbouxschrodinger}
		\left(\bigtriangleup -q_{\frac{1}{f}}(z)\right)v(z)= 0, \quad \mbox{ with } q_{\frac{1}{f}}= 2\frac{f_x^2+f_y^2}{f^2}-q_f = f\bigtriangleup\left(\frac{1}{f}\right),
	\end{equation} 
where $f_x=\frac{\partial f}{\partial x}$ and $f_y=\frac{\partial f}{\partial y}$.
\end{theorem}
The potential $q_{\frac{1}{f}}$ is called the Darboux transformation of $q_f$ \cite{darbouxorigin}. Denote the space of solutions of the main Vekua equation (\ref{mainvekua}) by
\begin{equation}\label{solspacevekua}
	\operatorname{V}_f(\Omega; \mathbb{B}):= \left\{ W\in C^1(\Omega; \mathbb{B})\, |\, W \mbox{ is solution of (\ref{mainvekua})}\right\}. 
\end{equation}	

\begin{theorem}[\cite{CamposBicomplex}, Th.13]\label{theoremvekuaclosed}
	The space $\operatorname{V}_f(\Omega; \mathbb{B})$ is closed in $C(\Omega; \mathbb{B})$.
\end{theorem}

\subsection{The pseudoanalytic Bergman space}

We introduce the pseudoanalytic Bergman space associated to the main Vekua equation (\ref{mainvekua}) as follows
\begin{equation}\label{bergmanpseudo}
	\mathcal{A}_f^2(\Omega; \mathbb{B}):= \left\{W\in \operatorname{V}_f(\Omega; \mathbb{B}) \, |\, W\in L_2(\Omega; \mathbb{B}) \right\}.
\end{equation}

When $f\equiv 1$ we obtain the analytic Bergman space $\mathcal{A}_2(\Omega; \mathbb{B})$. The following proposition generalized the proof given in \cite{delgadoleblond} for the complex pseudoanalytic Bergman space, that is, with $f$ real-valued (see also \cite{camposbergman} for the proof in the case of the real-valued Bergman pseudoanalytic space).

\begin{theorem}\label{theorembergmanspacecomplete}
	Let $\Omega\subset \mathbb{C}$ be a bounded domain and $f\in C^1(\overline{\Omega})$ that does not vanish in $\Omega$. Then $\mathcal{A}_f^2(\Omega; \mathbb{B})$ is a complex Hilbert space.
\end{theorem}

\begin{proof}
	Consider the Bicomplex Theodorescu operator (introduced in \cite{CamposBicomplex}) given by
	\[
	T_{\Omega}\Phi(z):=P^+B_{\Omega}\Phi^{+}(z)+P^{-}A_{\Omega}\Phi^{+}(z),\quad \Phi\in L_2(\Omega; \mathbb{B}),
	\] where $A_{\Omega}g(z):=\displaystyle \frac{1}{\pi}\iint_{\Omega}\frac{g(\zeta)}{\zeta-z}dA_{\zeta}$ and $B_{\Omega}g(z):=\displaystyle \frac{1}{\pi}\iint_{\Omega}\frac{g(\zeta)}{\zeta^*-z^*}dA_{\zeta}$ for $g\in L_2(\Omega)$. It is known that $A_{\Omega}, B_{\Omega}\in \mathcal{B}(L_2(\Omega), W^{1,2}(\Omega))$ (see, e.g., \cite[Prop. 5.2.1]{leblont}). Thus, $T_\Omega\in \mathcal{B}\left(L_2(\Omega; \mathbb{B}), W^{1,2}(\Omega; \mathbb{B})\right)$. Furthermore, it is known that $\frac{\partial}{\partial z}A_{\Omega}g=g$ and $\frac{\partial}{\partial z^*}B_{\Omega}g=g$ for $g\in L_2(\Omega)$ (\cite[Prop. 5.2.1]{leblont}). Hence
	\[
	\overline{\boldsymbol{\partial}}T_{\Omega}W=P^+\frac{\partial}{\partial z^*}B_{\Omega}W^+ +P^{-}\frac{\partial}{\partial z}A_{\Omega}W^{-}=W \quad \mbox{ for   }\; W\in L_2(\Omega; \mathbb{B}).
	\]  By the Weyl Lemma \cite[Ch. 18, Cor. 4.11]{conway2}, if $W\in L_2(\Omega;\mathbb{B})$ satisfies $\overline{\boldsymbol{\partial}}W=0$ in the weak sense (that is, $\frac{\partial W^+}{\partial z}=\frac{\partial W^{-}}{\partial z^*}=0$ in the weak sense), then $W\in \operatorname{Hol}(\Omega; \mathbb{B})$.  Take $W\in L_2(\Omega; \mathbb{B})$ and consider
	\[
	H=W-T_{\Omega}[\alpha_f\overline{W}], \quad \mbox{where  } \alpha_f:=\frac{\overline{\boldsymbol{\partial}}f}{f}.
	\]
	Since $f\in C^1(\overline{\Omega})$, $\alpha_f\overline{W}\in L_2(\Omega; \mathbb{B})$ and hence $H\in L_2(\Omega; \mathbb{B})$. Suppose that $W\in \mathcal{A}_f^2(\Omega; \mathbb{B})$. Then
	\begin{equation}\label{auxiliar1}
		\overline{\boldsymbol{\partial}}H=\overline{\boldsymbol{\partial}}W-\alpha_f\overline{W}=0.
	\end{equation}
	Thus, $H\in \mathcal{A}_2(\Omega; \mathbb{B})$. Reciprocally, if $H\in \mathcal{A}^2(\Omega; \mathbb{B})$, Eq. (\ref{auxiliar1}) implies that $W\in \mathcal{A}_f^2(\Omega; \mathbb{B})$. Now consider a sequence $\{W_n\}\subset \mathcal{A}_f^2(\Omega; \mathbb{B})$ such that $W_n\rightarrow W$ in $L_2(\Omega; \mathbb{B})$ for some $W\in L_2(\Omega; \mathbb{B})$. Then $H_n=W_n-T_{\Omega}[\alpha_f\overline{W}_n]\in \mathcal{A}^2(\Omega; \mathbb{B})$. Since $T_{\Omega}\in \mathcal{B}(L_2(\Omega; \mathbb{B}))$, then $H_n\rightarrow H=W-T_{\Omega}[\alpha_f\overline{W}]$ in $L_2(\Omega; \mathbb{B})$. Hence $H\in \mathcal{A}^2(\Omega; \mathbb{B})$ which implies that $W\in \mathcal{A}_f^2(\Omega; \mathbb{B})$.
\end{proof}
\begin{remark}\label{remarkseparablebergman}
	Since $\mathcal{A}_f^2(\Omega; \mathbb{B})$ is a subspace of the separable Hilbert space $L_2(\Omega; \mathbb{B})$, then $\mathcal{A}_f^2(\Omega; \mathbb{B})$ is a separable Hilbert space \cite[Prop. 3.25]{brezis}.
\end{remark}	
\begin{proposition}\label{propbergmankernelpseudo}
	Suppose that $f\in C^1(\overline{\Omega})\cap C^2(\Omega)\cap  W^{2,\infty}(\Omega)$. For any $K\subset \Omega$ there exists a constant $C_K>0$ such that
	\begin{equation}\label{Bergmanpseudoconstant}
		\max_{z\in K}|W(z)|_{\mathbb{B}}\leqslant C_K\|W\|_{L_2(\Omega;\mathbb{B})}, \quad \forall W\in \mathcal{A}_f^2(\Omega;\mathbb{B}).
	\end{equation}
\end{proposition}
\begin{proof}
		Let $W\in \mathcal{A}_f^2(\Omega; \mathbb{B})$ and $K\subset \Omega$ be compact. Set $u=\operatorname{Sc}W$ and $v=\operatorname{Vec}W$. Note that $u,v\in L_2(\Omega)$. By Theorem \ref{theoremvekuaconjugate}, $u$ is a solution of (\ref{schrodingereq}) and $v$ a solution of (\ref{darbouxschrodinger}). Since $f\in W^{2,\infty}(\Omega)$, $q_f,q_{\frac{1}{f}}\in L_{\infty}(\Omega)$, and according to \cite{friedrichs} and \cite[Sec. 2]{mine2}, there exist constants $C_1,C_2>0$ such that 
\[
	\max_{z\in K}|u(z)|\leqslant C_1 \|u\|_{L_2(\Omega)}, \quad \mbox{and }\quad \max_{z\in K}|v(z)|\leqslant C_1 \|v\|_{L_2(\Omega)}.
	\]
	Taking $C_3=\max\{C_1,C_2\}$ we obtain that for any $z\in K$
	\[
	|W(z)|_{\mathbb{B}}\leqslant |u(z)|+|v(z)|\leqslant C_3(\|u\|_{L_2(\Omega)}+\|v(z)\|_{L_2(\Omega)})\leqslant 2C_3 \|W\|_{L_2(\Omega; \mathbb{B})}.
	\]
	Thus, $\displaystyle \max_{z\in K}|W(z)|_{\mathbb{B}}\leqslant C_K\|W\|_{L_2(\Omega; \mathbb{B})}$, with $C_K=2C_3$. 
\end{proof}	
	
	Under the conditions of Proposition \ref{propbergmankernelpseudo}, for any $z\in \Omega$, the functionals $\mathcal{A}_f^2(\Omega; \mathbb{B})\ni W \mapsto \operatorname{Sc}W(z), \operatorname{Vec}W(z)\in \mathbb{C}$ are bounded. Then by the Riesz representation theorem there exist kernels $K_z, L_z\in \mathcal{A}_f^2(\Omega; \mathbb{B})$ such that 
	\[
	\operatorname{Sc}W(z)=\iint_{\Omega}\langle W(\zeta), K_z(\zeta)\rangle_{\mathbb{B}}dA_{\zeta}, \quad \operatorname{Vec}W(z)=\iint_{\Omega}\langle W(\zeta), L_z(\zeta)\rangle_{\mathbb{B}}dA_{\zeta}.
	\]
	These results are similar to those obtained in \cite{camposbergman} for the complex-valued Bergman pseudoanalytic space.

\begin{remark}\label{remarkbergmankernel2}	
	The kernels $K_z(\zeta)$ and $L_z(\zeta)$ satisfies the relations
	\begin{equation}\label{relationsbergmankernel}
		\operatorname{Sc}K_z(\zeta)=(\operatorname{Sc}K_{\zeta}(z))^{*}, \quad \operatorname{Vec}L_z(\zeta)=(\operatorname{Vec}L_{\zeta}(z))^{*}, \quad \operatorname{Sc}L_z(\zeta)= (\operatorname{Vec}K_{\zeta}(z))^{*}
	\end{equation}
Indeed, since $K_z\in \mathcal{A}_f^2(\Omega; \mathbb{B})$ we have
\[
\operatorname{Sc}K_z(\zeta)= \langle K_z, K_{\zeta}\rangle_{L_2(\Omega;\mathbb{B})}= \langle K_{\zeta}, K_z\rangle_{L_2(\Omega;\mathbb{B})}^{*}=\left(\operatorname{Sc}K_{\zeta}\right)^*(z).
\]
The other two equalities are proved analogously.
\end{remark}
Using (\ref{relationsbergmankernel}), for any $W\in \mathcal{A}_f^2(\Omega; \mathbb{B})$ we obtain  the following relation
\begin{align*}
	W(z)& = \operatorname{Sc}W(z)+\mathbf{j}\operatorname{Vec}W(z)\\
	  & = \iint\limits_{\Omega}\left( \operatorname{Sc}W(\zeta)(\operatorname{Sc}K_z(\zeta))^{*}+\operatorname{Vec}W(\zeta)(\operatorname{Vec}K_z(\zeta))^{*}\right)dA_{\zeta} \\
	  &\quad + \mathbf{j}\iint\limits_{\Omega}\left( \operatorname{Sc}W(\zeta)(\operatorname{Sc}L_z(\zeta))^{*}+\operatorname{Vec}W(\zeta)(\operatorname{Vec}L_z(\zeta))^{*}\right)dA_{\zeta}\\
      & = \iint\limits_{\Omega}\left( \operatorname{Sc}W(\zeta)\operatorname{Sc}K_{\zeta}(z)+\operatorname{Vec}W(\zeta)\operatorname{Sc}L_{\zeta}(z)\right)dA_{\zeta} \\
     &\quad + \mathbf{j}\iint\limits_{\Omega}\left( \operatorname{Sc}W(\zeta)\operatorname{Vec}K_{\zeta}(z)+\operatorname{Vec}W(\zeta)\operatorname{Vec}L_{\zeta}(z)\right)dA_{\zeta},
\end{align*}
from where we obtain the relation
\begin{equation}\label{bergmankernel1}
	W(z)=\iint\limits_{\Omega} \left( \operatorname{Sc}W(\zeta)K_{\zeta}(z)+\operatorname{Vec}W(\zeta)L_{\zeta}(z)\right)dA_{\zeta}.
\end{equation}
For all $z,\zeta\in \Omega$, define $K(z,\zeta)=K_{\zeta}(z)$ and $L(z,\zeta)=L_{\zeta}(z)$. Following \cite{camposbergman}, we introduce the next definition of the Bergman kernel.

\begin{definition}
	The Bergman kernel of the space $\mathcal{A}_f^2(\Omega; \mathbb{B})$ with coefficient $A\in \mathbb{B}$ is defined by
	\begin{equation}\label{Bergmankernelvekua}
		\mathscr{K}_{\Omega}^f(A;z,\zeta):= \operatorname{Sc}(A) K(z,\zeta)+\operatorname{Vec}(A)L(z,\zeta), \quad z,\zeta\in \Omega. 
	\end{equation}
\end{definition}
	Then we have the reproducing property
\begin{equation}\label{reproducingproperty}
	W(z)=\iint\limits_{\Omega}\mathscr{K}_{\Omega}^f(W(\zeta);z,\zeta)dA_{\zeta}.
\end{equation}
This definition for the Bergman kernel of the Vekua equation was introduced in \cite{camposbergman} for the case when $f$ is real valued (and $\mathcal{A}_f^2(\Omega; \mathbb{C})$ is a {\it real} Hilbert space).

\begin{remark}\label{remarkbergmankernelexpasion}
	Note that $\mathscr{K}_{\Omega}^f(A; \cdot, \zeta)\in \mathcal{A}_f^2(\Omega; \mathbb{B})$ for all $A\in \mathbb{B}$, $\zeta \in \Omega$. Since $\mathcal{A}_f^2(\Omega; \mathbb{B})$ is a separable Hilbert space, it admits an orthonormal basis of the form $\{\Phi_n\}_{n=0}^{\infty}$ \cite[Th. 5.11]{brezis}. Then
	\begin{align*}
		\mathscr{K}_{\Omega}^f(A; z, \zeta)=\sum_{n=0}^{\infty}\left\langle \mathscr{K}_{\Omega}^f(A; \cdot, \zeta), \Phi_n \right\rangle_{L_2(\Omega;\mathbb{B})}\Phi_n(z),
	\end{align*}
and
\begin{align*}
	\left\langle \mathscr{K}_{\Omega}^f(A; \cdot, \zeta), \Phi_n \right\rangle_{L_2(\Omega;\mathbb{B})} & = \left\langle \operatorname{Sc}(A)K(\cdot, \zeta)+\operatorname{Vec}(A)L(\cdot, \zeta), \Phi_n(\cdot)\right\rangle_{L_2(\Omega; \mathbb{B})} \\
	& =\operatorname{Sc}(A) \left\langle K_{\zeta}(\cdot), \Phi_n(z)\right\rangle_{L_2(\Omega; \mathbb{B})}+\operatorname{Vec}(A)\left\langle L_{\zeta}(\cdot), \Phi_n(\cdot)\right\rangle_{L_2(\Omega; \mathbb{B})}\\
	&= \operatorname{Sc}(A)\left(\operatorname{Sc}\Phi_n(\zeta)\right)^{*}+\operatorname{Vec}(A)\left(\operatorname{Vec}\Phi_n(\zeta)\right)^{*}=\langle  A, \Phi_n(\zeta)\rangle_{\mathbb{B}}.
\end{align*}
Thus,
\begin{equation}\label{bergmankernelexpansion}
	\mathscr{K}_{\Omega}^f(A; z, \zeta)=\sum_{n=0}^{\infty}\langle  A, \Phi_n(\zeta)\rangle_{\mathbb{B}}\Phi_n(z).
\end{equation}
Since $\Omega$ is bounded, the constant functions belong to $\mathcal{A}_f^2(\Omega; \mathbb{B})$, so the coefficients of (\ref{bergmankernelexpansion}) lie in $\ell_2(\mathbb{N}_0)$ and the series (\ref{bergmankernelexpansion}) converges in the $L_2(\Omega;\mathbb{B})$-norm. By Proposition \ref{propbergmankernelpseudo}, series (\ref{bergmankernelexpansion}) converges in the variable $z$ uniformly on compact subsets of $\Omega$.
\end{remark}
\begin{remark}
	Since $\mathcal{A}_f^2(\Omega;\mathbb{B})$ is a closed subspace of $L_2(\Omega;\mathbb{B})$, we have the orthogonal decomposition $L_2(\Omega;\mathbb{B})=\mathcal{A}_f^2(\Omega;\mathbb{B})\oplus\left(\mathcal{A}_f^2(\Omega;\mathbb{B})\right)^{\perp}$. Let $\mathbf{P}_{\Omega}: L_2(\Omega; \mathbb{B})\rightarrow \mathcal{A}_f^2(\Omega;\mathbb{B})$ the corresponding orthogonal projection, that we call the Bergman projection. Similar to the complex-valued case \cite[Prop 1.4]{camposbergman}, we have that the Bergman projection can be written as
	\begin{equation}
		\mathbf{P}_{\Omega}\Phi(z)= \iint_{\Omega}\mathscr{K}_{\Omega}^f(\Psi(\zeta),z,\zeta)dA_{\zeta}, \qquad \forall \Psi \in L_2(\Omega;\mathbb{B}).
	\end{equation}
Indeed, let $W=\mathbf{P}_{\Omega}\Psi(z)$ for $\Psi\in L_2(\Omega; \mathbb{B})$. Hence $W\in \mathcal{A}_f^2(\Omega; \mathbb{B})$ and
we can write $W$ using (\ref{reproducingproperty}). Now, in the construction of (\ref{reproducingproperty}), we obtain the equality
\begin{equation*}
	\iint_{\Omega}\mathscr{K}_{\Omega}(W(\zeta),z,\zeta)= \langle W, K_{z}\rangle_{L_2(\Omega; \mathbb{B})}+\mathbf{j}\langle W, L_{z}\rangle_{L_2(\Omega; \mathbb{B})}.
\end{equation*}
Notice that the right-hand side of this equality is well-defined even for $L_2$-functions. Since the orthogonal projection is a self-adjoint operator \cite[Prop. 3.3]{conwayfunctional}, we have
\begin{align*}
	\mathbf{P}_{\Omega}\Psi(z) & = \langle \mathbf{P}_{\Omega}\Psi, K_{z}\rangle_{L_2(\Omega; \mathbb{B})}+\mathbf{j}\langle \mathbf{P}_{\Omega}\Psi, L_{z}\rangle_{L_2(\Omega; \mathbb{B})} =  \langle \Psi, \mathbf{P}_{\Omega}K_{z}\rangle_{L_2(\Omega; \mathbb{B})}+\mathbf{j}\langle \Psi, \mathbf{P}_{\Omega}L_{z}\rangle_{L_2(\Omega; \mathbb{B})}\\
	 &=\langle \Psi, K_{z}\rangle_{L_2(\Omega; \mathbb{B})}+\mathbf{j}\langle \Psi, L_{z}\rangle_{L_2(\Omega; \mathbb{B})}= \iint_{\Omega}\mathscr{K}_{\Omega}^f(\Psi(\zeta),z,\zeta)dA_{\zeta}.
\end{align*}
\end{remark}
\begin{remark}
	In the case $f\equiv 1$ we obtain the Bicomplex analytic Bergman space, and $\mathbf{j}W\in \mathcal{A}^2(\Omega;\mathbb{B})$ if $W\in \mathcal{A}^2(\Omega;\mathbb{B})$. It is not difficult to see that $\langle \mathbf{j}W, V\rangle_{L_2(\Omega;\mathbb{B})}=-\langle W,\mathbf{j} V\rangle_{L_2(\Omega;\mathbb{B})}$ for $W,V\in L_2(\Omega;\mathbb{B})$, and that $\operatorname{Sc}(\mathbf{j}W)=-\operatorname{Vec}W$,$\operatorname{Vec}(\mathbf{j}W)=\operatorname{Sc}W$. Hence
	\begin{align*}
		\operatorname{Sc}L(z,\zeta)& = \operatorname{Sc}L_{\zeta}(z)= \operatorname{Vec}\left(\mathbf{j}L_{\zeta}(z)\right) = \langle \mathbf{j}L_{\zeta}, L_z\rangle_{L_2(\Omega; \mathbb{B})}=-\langle L_{\zeta},\mathbf{j} L_z\rangle_{L_2(\Omega; \mathbb{B})}\\
		& = \left(\operatorname{Vec}(\mathbf{j}L_z(\zeta))\right)^*=
		-\left(\operatorname{Sc}L(\zeta,z)\right)^{*}.
	\end{align*}
In a similar way $\operatorname{Sc}K(z,\zeta)=\left(\operatorname{Vec}L(\zeta,z)\right)^{*}$. Using these equalities together with (\ref{relationsbergmankernel}) we obtain
\begin{align*}
	L(z,\zeta) & = \operatorname{Sc}L(z,\zeta)+\mathbf{j}\operatorname{Vec}L(z,\zeta)=-\left(\operatorname{Sc}L(\zeta,z)\right)^{*}+\mathbf{j}\left(\operatorname{Sc}K(\zeta,z)\right)^{*}\\
	&= -\operatorname{Vec}K(z,\zeta)+\mathbf{j}\operatorname{Sc}K(z,\zeta)= \mathbf{j}K(z,\zeta).
\end{align*}
Substituting this equality in (\ref{Bergmankernelvekua}) and (\ref{reproducingproperty}) we obtain that $K(z,\zeta)$ is the analytic Bergman kernel obtained in Subsection \ref{secbicomplexanalytic}. 
\end{remark}

\section{Transmutation operators}
\subsection{The Vekua equation in polar coordinates}

From now on, we assume that $\Omega$ is a bounded star-shaped domain with respect to $z=0$, that is, $0\in \Omega$ and for any $z\in \Omega$, the segment $[0,z]:=\{tz\, |\, 0\leqslant t\leqslant 1\}$ belongs to $\Omega$. In this case, we denote $\varrho_{\Omega}:=\displaystyle  \sup_{z\in \Omega}|z|$. Suppose that $q\in C^1[0, \varrho_{\Omega}]$ is scalar complex-valued and consider the radial Schr\"odinger equation
\begin{equation}\label{radialschrodinger}
	\left(\bigtriangleup-q(|z|)\right)u(z)=0, \quad z\in \Omega.
\end{equation}
In this case, the potential only depends on the radial component $r=|z|$. In polar coordinates $z=re^{i\theta}$, Eq. (\ref{radialschrodinger}) has the form
\begin{equation}
	\left(\frac{\partial^2}{\partial r^2}+\frac{1}{r}\frac{\partial}{\partial r}+\frac{1}{r^2}\frac{\partial^2}{\partial \theta^2}-q(r)\right)u(r,\theta)=0, \quad (r,\theta)\in (0,\varrho_{\Omega})\times [0, 2\pi].
\end{equation}
Suppose that there exists a non-vanishing scalar complex-valued solution of (\ref{radialschrodinger}) $\tilde{f}\in C^2(\overline{\Omega})$ that only depends on the radial component and take $f(r):=\widetilde{f}(|z|)$. A direct computation shows that if $f\in C^2[0, \varrho_{\Omega}]$ is a solution of the Bessel equation
\begin{equation}
	f''(r)+\frac{1}{r}f'(r)-q(r)f(r)=0, \quad 0<r<\varrho_{\Omega}.
\end{equation} 
From now on, we identify $f$ with $\widetilde{f}$.
\begin{remark}\label{CRinpolar}
 The Bicomplex Cauchy-Riemann operators can be written in the form
\begin{equation}\label{CRinpolarcoordinates}
	\overline{\boldsymbol{\partial}}=\frac{e^{\mathbf{j}\theta}}{2}\left(\frac{\partial}{\partial r}+\frac{\mathbf{j}}{r}\frac{\partial}{\partial \theta}\right), \quad \boldsymbol{\partial}=\frac{e^{-\mathbf{j}\theta}}{2}\left(\frac{\partial}{\partial r}-\frac{\mathbf{j}}{r}\frac{\partial}{\partial \theta}\right)
\end{equation}
\end{remark}

In consequence we obtain
\begin{equation}\label{coeffradialmain}
	\frac{\overline{\boldsymbol{\partial}}f(z)}{f(z)}= \frac{e^{\mathbf{j}\theta}}{2}\frac{f'(r)}{f(r)}.
\end{equation}
Thus, the main Vekua equation associated to $f$ can be written as
\begin{equation}\label{radialVekuaEq1}
	\frac{e^{\mathbf{j}\theta}}{2}\left(\frac{\partial}{\partial r}W+\frac{\mathbf{j}}{r}\frac{\partial}{\partial \theta}W-\frac{f'(r)}{f(r)}\overline{W}\right)=0.
\end{equation}
that is reduced to
\begin{equation}\label{radialVekuaeq}
	\frac{\partial}{\partial r}W+\frac{\mathbf{j}}{r}\frac{\partial}{\partial \theta}W-\frac{f'(r)}{f(r)}\overline{W}=0, \quad (r,\theta)\in (0,\varrho_{\Omega})\times [0,2\pi].
\end{equation}
Eq. (\ref{radialVekuaeq}) will be called the {\bf radial (main) Vekua equation}. If $W=u+\mathbf{j}v$ with $u=\operatorname{Sc}W,v=\operatorname{Vec}W$, Eq. (\ref{radialVekuaeq}) is equivalent to the system
\begin{eqnarray}
	f\frac{\partial}{\partial r}\left(\frac{u}{f}\right) & = & \frac{1}{r}\frac{\partial v}{\partial \theta} \label{CRsystem1}\\
	\frac{1}{f}\frac{\partial}{\partial r}\left(f v\right) & = & -\frac{1}{r}\frac{\partial v}{\partial \theta}. \label{CRsystem2}
\end{eqnarray}

On the other hand, since $\frac{\partial f}{\partial x}=\cos(\theta)f'(r)$ and $\frac{\partial f}{\partial y}=\sin(\theta)f'(r)$, the Darboux potential is given by
\[
q_{\frac{1}{\tilde{f}}}(z)=2\left(\frac{f'(r)}{f(r)}\right)^2-q(r).
\]
Then  $q_f(r)=\frac{1}{f(r)}\left(f''(r)+\frac{1}{r}f'(r)\right)$ and $q_{\frac{1}{f}}(r)= 2\left(\frac{f'(r)}{f(r)}\right)^2-q_f(r)$. Hence, $u$ and $v$ satisfies system (\ref{CRsystem1}),(\ref{CRsystem2}) iff are solutions of the radial Schr\"odinger equation with potentials $q_f$ and $q_{\frac{1}{f}}$, respectively.

\subsection{Transmutation operators for the radial Schr\"odinger equation}

Suppose that $q_f\in C^1[0,\varrho_{\Omega}]$. Denote by $\mathbf{S}_f:= -\bigtriangleup+q_f(r)$ the radial Schr\"odinger operator with potential $q_f$, and define
\[
\operatorname{Sol}^{\mathbf{S}_f}(\Omega):= \{ u\in C^2(\Omega)\, |\, u \mbox{ is solution of }  \mathbf{S}_fu=0 \; \mbox{ in } \Omega\}.
\]
The following result establishes that any solution $u\in \operatorname{Sol}^{\mathbf{S}_f}(\Omega)$ can be written as the image of a harmonic function $h\in \operatorname{Har}(\Omega)$ under the action of an integral operator.

\begin{theorem}[\cite{gilbertatkinson}]\label{thereomtransformationoperator}
	Let $q_f\in C^1[0,\varrho_{\Omega}]$.
	\begin{itemize} 
	\item[(i)]  There exists a kernel $G^f\in C^{2}([0,\varrho_{\Omega}]\times [0,1])$ satisfying the partial differential equation
	\begin{equation}\label{PDEkernel1}
		r(G_rr(r,t)-q_f(r)G^f(r,t))-G_r^f(r,t)+2(1-t)G_{rt}^f(r,t)=0, \quad (r,t)\in [0,\varrho_{\Omega}]\times [0,1],
	\end{equation}
with the initial conditions
\begin{equation}\label{initialconditionskernelG}
	G^f(r,0)=\int_0^r\tau q(\tau)d\tau \; \forall r\in [0,\varrho_{\Omega}]; \quad G^f(0,t)=0, \; \forall t\in [0,1].
\end{equation}
\item[(ii)] For any solution $u\in  \operatorname{Sol}^{\mathbf{S}_f}(\Omega)$, there exists a unique $h\in \operatorname{Har}(\Omega)$ such that
\begin{equation}\label{transmutation1}
	\mathbf{T}_fu(z)= h(z)+\int_0^1\sigma G^f(r, 1-\sigma^2)h(\sigma^2z)d\sigma.
\end{equation}
Then the operator $\mathbf{T}_f$ transforms $\operatorname{Har}(\Omega)$ onto $\operatorname{Sol}^{\mathbf{S}_f}(\Omega)$.
\end{itemize}
\end{theorem}
 Furthermore, the operator $\mathbf{T}_f$ is a {\it transmutation operator}, in the sense of the following definition.
 \begin{definition}[\cite{transmutation1}]\label{deftransmop}
 	Let $E$ be a topological vector space, $E_{1}\subset
 	E$ a linear subspace (not necessarily closed), and $\mathbf{A},\mathbf{B}%
 	:E_{1}\rightarrow E$ linear operators. A linear invertible operator
 	$\mathbf{T}:E\rightarrow E$, such that $E_{1}$ is $\mathbf{T}$-invariant, is
 	said to be a \textbf{transmutation operator} for the pair of operators
 	$\mathbf{A},\mathbf{B}$, if the following conditions are fulfilled:
 	
 	\begin{enumerate}
 		\item Both the operator $\mathbf{T}$ and its inverse $\mathbf{T}^{-1}$ are
 		continuous in $E$.
 		
 		\item The following operator equality is valid
 		\begin{equation}
 			\mathbf{AT}=\mathbf{TB}\quad\mbox{ in }E_{1}.\label{transmutationdef}%
 		\end{equation}
 		
 	\end{enumerate}
 \end{definition}
 In this case, $E=C(\Omega)$ with the topology of the uniform convergence on compact subsets and $E_1=C^2(\Omega)$.
 \begin{theorem}[\cite{mine1}]\label{thtransmutationradial1}
 	The operator $\mathbf{T}_f$ and its inverse are continuous in the Fr\'echet space $C(\Omega)$, and $\mathbf{T}_f$ is a transmutation operator for the pair of operators $r^2\mathbf{S}_f$ and $r^2\bigtriangleup$. Explicitly, the following relation holds
 	\begin{equation}\label{transrelation1}
 		r^2\mathbf{S}_f \mathbf{T}_fu=\mathbf{T}_fr^2\bigtriangleup u, \quad \forall u\in C^2(\Omega).
 	\end{equation}
 \end{theorem}

The corresponding transmutation operator for the Darboux transformed operator $\mathbf{S}_{\frac{1}{f}}= \bigtriangleup-q_{\frac{1}{f}}(r)$ is denoted by $\mathbf{T}_{\frac{1}{f}}$. 

Let $W\in \operatorname{V}_f(\Omega;\mathbb{B})$. By Theorem \ref{theoremvekuaconjugate}, $u=\operatorname{Sc}W\in \operatorname{Sol}^{\mathbf{S}_f}(\Omega)$ and $v=\operatorname{Vec}W\in \operatorname{Sol}^{\mathbf{S}_{\frac{1}{f}}}(\Omega)$. Due to  Theorem \ref{thereomtransformationoperator}(ii), there exists $h_1,h_2\in \operatorname{Har}(\Omega)$ such that
\begin{equation}\label{transformationvekua}
	W=\mathbf{T}_fh_1+\mathbf{j}\mathbf{T}_{\frac{1}{f}}h_2.
\end{equation}

The idea is to use an operator of the form 
\begin{equation}
\mathcal{T}_f:= 	\mathbf{T}_f \operatorname{Sc}+\mathbf{j}\mathbf{T}_{\frac{1}{f}} \operatorname{Vec}
\end{equation}
to transform the space $\operatorname{Hol}(\Omega; \mathbb{B})$ onto $\operatorname{V}_f(\Omega;\mathbb{B})$. Similarly, we define $\mathcal{T}_{\frac{1}{f}}$. Operators of form (\ref{transformationvekua}) have been used to transform holomorphic functions into solutions of the Vekua equation for the case when $f$ is a function that only depends on the first component $x=\operatorname{Re}z$ (see \cite{hugo,transhiperbolic}). Before studying the transmutation property of $\mathcal{T}_f$ we establish their analytical properties. 
\begin{remark}\label{blinearity}
	It is clear that the operator $\mathcal{T}_f$ is linear in the complex space $C(\Omega; \mathbb{B})$. Consider $C(\Omega;\mathbb{B})$ as a $\mathbb{B}$-module. Take $A\in \mathbb{B}$, $W\in C(\Omega;\mathbb{B})$ and $a=\operatorname{Sc}A$, $b=\operatorname{Vec}A$, $u=\operatorname{Sc}W$, $v=\operatorname{Vec}W$. We have
	\begin{align*}
		\mathcal{T}_f[AW]& = \mathcal{T}_f[(a+\mathbf{j}b)(u+\mathbf{j}v)] = \mathcal{T}_f[(au-bv)+\mathbf{j}(bu+av)]\\
		&= \mathbf{T}_f[au-bv]+\mathbf{j}\left(\mathbf{T}_{\frac{1}{f}}[bu+av]\right)= a\mathbf{T}_fu-b\mathbf{T}_fv+\mathbf{j}\left(b\mathbf{T}_{\frac{1}{f}}u+a\mathbf{T}_{\frac{1}{f}}u\right)\\
		&= a\mathcal{T}_fW+\mathbf{j}b\mathcal{T}_{\frac{1}{f}}W.
	\end{align*}  
Thus, the operator $\mathcal{T}_f$ fails to be $\mathbb{B}$-linear. However, we have the relation 
\begin{equation}\label{noBlinearity}
	\mathcal{T}_f[AW]=\operatorname{Sc}(A)\mathcal{T}_fW+\mathbf{j}\operatorname{Vec}(A)\mathcal{T}_{\frac{1}{f}}W,\quad A\in \mathbb{B}, W\in C(\Omega; \mathbb{B}).
\end{equation}
In particular
\begin{equation}\label{relationmultiplicationbyj}
	\mathcal{T}_f[\mathbf{j}W]=\mathbf{j}\mathcal{T}_{\frac{1}{f}}W.
\end{equation}
For this reason, we consider $V_f(\Omega; \mathbb{B})$ as a complex linear space.
\end{remark}

\begin{proposition}\label{propertiesopt}
	The operator $\mathcal{T}: C(\Omega; \mathbb{B})\rightarrow C(\Omega; \mathbb{B})$ is continuous and invertible, and its inverse is given by
	\begin{equation}\label{inverseofT}
		\mathcal{T}_f^{-1}= \mathbf{T}_f^{-1} \operatorname{Sc}+\mathbf{j}\mathbf{T}_{\frac{1}{f}}^{-1} \operatorname{Vec}.
	\end{equation}
The inverse $\mathcal{T}_f^{-1}$ is also continuous in $C(\Omega; \mathbb{B})$. The same property is valid for $\mathcal{T}_{\frac{1}{f}}$.
\end{proposition}
\begin{proof}
	since $\mathbf{T}_f^{-1}$ and $\mathbf{T}_{\frac{1}{f}}$ exists and are continuous \cite[Sec. II]{mine1}, a simple computation shows that the inverse of $\mathcal{T}_f$ is given by (\ref{inverseofT}). The continuity of $\mathcal{T}_f$ and $\mathcal{T}_f^{-1}$ in $C(\Omega; \mathbb{B})$ is due to the continuity of $\mathbf{T}_f, \mathbf{T}_{\frac{1}{f}}$, $\mathbf{T}_f^{-1},$ and $\mathbf{T}_{\frac{1}{f}}^{-1}$ in $C(\Omega)$. Similar for $\mathcal{T}_{\frac{1}{f}}$.
\end{proof}

\begin{remark}\label{remarktransmutationc2}
	The operators $\mathbf{T}_f$ and $\mathbf{T}_{\frac{1}{f}}$ satisfy that $\mathbf{T}_{f}(C^2(\Omega))=C^2(\Omega)=\mathbf{T}_{\frac{1}{f}}(C^2(\Omega))$ (see \cite[Sec. 2]{mine1}). Thus, $\mathcal{T}_{f}\left(C^2(\Omega; \mathbb{B})\right)=C^2(\Omega; \mathbb{B})$. Similar for $\mathcal{T}_{\frac{1}{f}}$.
\end{remark}

\subsection{A representation for the transmutation operator of the Darboux equation}

Consider the radial component of the Schr\"odinger operator $\mathbf{S}_f$, that is given by
\begin{equation}
	\mathbf{L}_f:=\frac{d^2}{dr^2}+\frac{1}{r}\frac{d}{dr}-q_f(r)= \frac{1}{r}\frac{d}{dr}r\frac{d}{dr}-q_f(r).
\end{equation}
We assume that $f$ is normalized satisfying the initial conditions
\begin{equation}\label{initialconditionsf}
	f(0)=1, \quad \mbox{and } \lim_{r\rightarrow 0^+}rf'(r)=0.
\end{equation}
In such case, $\frac{1}{f}$ satisfies the same conditions (\ref{initialconditionsf}). 

\begin{example}\label{examplehelmholtz}
	Let $\kappa\in (0,1)$. Consider $\Omega=\mathbb{D}$ and $q(r)=-\kappa^2$ (The Helmholtz equation $\bigtriangleup u+\kappa^2u=0$). In this case, the regular solution of $f''+\frac{1}{r}f=-\kappa^2f$ in $[0,1]$ is given by $f(r)=J_0(\kappa r)$ (Bessel function of first kind). Since $\kappa<1$, $J_0(\kappa r)>0$ for $r\in [0,1]$. Also, $J_0(0)=1$ and $J'(\kappa r)= -\frac{\kappa^2r}{2}+o(r^2)$, $r\rightarrow 0^+$, and then $f$ satisfies (\ref{initialconditionsf}). In this case, the operator $\mathbf{T}_f$ is known explicitly \cite{vekua}
	\begin{equation*}
		\mathbf{T}_fh(r)=h(r)-\int_0^1\frac{\partial}{\partial \rho}J_0\left(\alpha r\sqrt{1-\sigma^2}\right)h(\sigma^2 r)d\sigma.
	\end{equation*}
Using \cite[pp. 361, Formula 9.1.28]{abramowitz}, $\frac{f'(r)}{f(r)}=-\kappa$, then the corresponding Vekua equation is given by $\overline{\boldsymbol{\partial}}W+\kappa \overline{W}=0$ in $\mathbb{D}$. Te Darboux potential is given by $q_{\frac{1}{f}}= 3\kappa^2$.
\end{example}
The Polya factorization of $\mathbf{L}_f$ is given by
\begin{equation}\label{polyaf}
	\mathbf{L}_f=\frac{1}{rf}\frac{d}{dr}rf^2\frac{d}{dr}\frac{1}{f}.
\end{equation}
In a similar way
\begin{equation}\label{polya1f}
	\mathbf{L}_{\frac{1}{f}}=\frac{f}{r}\frac{d}{dr}\frac{r}{f^2}\frac{d}{dr}f.
\end{equation}
In the case $f\equiv 1$ we obtain the radial part of the Laplacian, $\mathbf{L}_1=\frac{1}{r}\frac{d}{dr}r\frac{d}{r}$. If $\mathbf{L}$ is any of these operators, we use the notation $\widehat{\mathbf{L}}:=r^2\mathbf{L}$. Theorem \ref{thtransmutationradial1} can be formulated in terms of the radial operators as follows
\begin{equation}
	\widehat{\mathbf{L}}_f\mathbf{T}_f u(r)= \mathbf{T}_f\widehat{\mathbf{L}}_1u(r), \quad \forall u\in C^2[0,\varrho_{\Omega}]
\end{equation}
(see \cite[Sec. 7]{mine1}). Set $\mathbf{D}_f:=\frac{r}{f}\frac{d}{dr}f$. Hence we have the following factorization
\begin{equation}\label{factorizations}
	\widehat{\mathbf{L}}_f=\mathbf{D}_f \mathbf{D}_{\frac{1}{f}}, \quad \widehat{\mathbf{L}}_\frac{1}{f}= \mathbf{D}_{\frac{1}{f}}\mathbf{D}_f.
\end{equation}

\begin{remark}\label{remarkasymp}
	Let $\lambda\geqslant 0$ and consider the equation
	\begin{equation}\label{eigenvalueequation}
		\widehat{\mathbf{L}}_fu(\lambda,r)= \lambda^2 u(\lambda,r), \quad r\in (0, \varrho_{\Omega}].
	\end{equation}
	  Set $u(\lambda,r)=\frac{y(\lambda,r)}{\sqrt{r}}$. A direct computation shows that $u$ satisfies the equation
	  \begin{equation*}
	  	y''-\frac{\lambda^2-\frac{1}{4}}{r^2}y-q_f(r)y=0, \quad r\in (0, \varrho_{\Omega}],
	  \end{equation*}
  that can be written as the Perturbed Bessel equation
  \begin{equation}\label{perturbedbesselequation}
  	y''-\frac{\ell(\ell+1)}{r^2}y-q_f(r)y=0, \quad r\in (0, \varrho_{\Omega}],
  \end{equation}
with $\ell=\lambda-\frac{1}{2}\geqslant -\frac{1}{2}$. It is known that for potentials $q_f\in L_1(0, \varrho_{\Omega})$ satisfying the condition $\int_0^{\varrho_{\Omega}}r^{\mu}|q_f(r)|dr<\infty$ for some $\mu \in [0, \frac{1}{2}]$, there exists a unique solution $y\in C[0,\varrho_{\Omega}]\cap C^2(0,\varrho_{\Omega})$ that satisfies the asymptotic relations (see \cite{KTBesseltrans})
\begin{equation}\label{asymptBessel}
y(\lambda,r)\sim r^{\lambda+\frac{1}{2}}, \quad y'(\lambda,r)\sim \left(\lambda+\frac{1}{2}\right)r^{\lambda-\frac{1}{2}}, \quad r\rightarrow 0^+.
\end{equation}
In particular this is valid for $q_f\in C^1[0,\varrho_{\Omega}]$. Since $u=\frac{y}{\sqrt{r}}$ and $u'=-\frac{1}{2}\frac{u}{r}+\frac{y'}{\sqrt{r}}$, from (\ref{asymptBessel}) we obtain the following asymptotic relations
\begin{equation}\label{asymp2}
	u(\lambda,r) \sim r^{\lambda}, \quad u'(\lambda,r)\sim \lambda r^{\lambda-1}, \quad r\rightarrow 0^+.
\end{equation}
For this reason we call $u(\lambda,r)$ the regular solution of (\ref{eigenvalueequation}) at $r=0$. In this way, $f$ is the regular solution of (\ref{eigenvalueequation}) for the case $\lambda=0$.
\end{remark}
Suppose that $\lambda\geqslant 0$ and let $v(\lambda,r)$ be the regular solution at $r=0$ of $\widehat{\mathbf{L}}_{\frac{1}{f}}v(\lambda,r)=\lambda v(\lambda,r)$. By (\ref{factorizations}), $u(\lambda, r)=\mathbf{D}_fv(\lambda,r)$ is a solution of (\ref{eigenvalueequation}). Furthermore, a direct computation shows that
\begin{align*}
	u & = rv'+r\frac{f'}{f}v, \\
	u' & =rv''+v'-rq_{\frac{1}{f}}v+r\left(\frac{f'}{f}\right)^2v+r\frac{f'}{f}v'\\
	& = \frac{\lambda^2v}{r}+r\left(\frac{f'}{f}\right)^2v+r\frac{f'}{f}v'.
\end{align*}
Since $v^{(k)}(r)\sim \frac{d^k}{dr^k}r^{\lambda}$, $r\rightarrow 0^+$ for $k=0,1$, and $f$ satisfies (\ref{initialconditionsf}), we have that $u^{(k)}(r)\sim \lambda\frac{d^k}{dr^k}r^{\lambda}$, $r\rightarrow 0^+$ for $k=0,1$. Thus, $u=\lambda \tilde{u}$, where $\tilde{u}$ is the regular solution at $r=0$ of (\ref{eigenvalueequation}).
 A right inverse of $\mathbf{D}_f$ is given by $\mathbf{I}_fu(r)=\frac{1}{f(r)}\left(\int_0^r\frac{f(s)u(s)}{s}ds+C\right)$, with $C\in \mathbb{C}$. Note that when $f\equiv 1$, the Darboux transformation is nothing but $\mathbf{D}_1=r\frac{d}{dr}$. Following the ideas from \cite{transmutationdarboux}, we look to show that the transmutation operator $\mathbf{T}_{\frac{1}{f}}$ can be written in a suitable form as the composition $\mathbf{T}_{\frac{1}{f}}= \mathbf{I}_f \mathbf{T}_f\mathbf{D}_1$, in some appropriate subspace, in such away that the following diagram commutes.
\[
\xymatrix{
\widehat{\mathbf{L}}_1+\lambda \ar[r]^{\mathbf{T}_f}& \widehat{\mathbf{L}}_f+\lambda \ar[d]^{\mathbf{I}_f}  \\
\widehat{\mathbf{L}}_1+\lambda \ar[u]^{\mathbf{D}_1}\ar[r]^{\mathbf{T}_{\frac{1}{f}}} & \widehat{\mathbf{L}}_{\frac{1}{f}}+\lambda 
}
\]

The reason of use the composition with the integral operator $\mathbf{I}_f$ instead of looking for the relation $\mathbf{T}_\frac{1}{f}=\mathbf{D}_{\frac{1}{f}}\mathbf{T}_f$ is to obtain a bounded operator.

\begin{remark}\label{remarktransmutespowers}
	Given $\lambda\geqslant 0$,  the regular solution at $r=0$ of $\widehat{\mathbf{L}}_1h(\lambda,r)=\lambda^2h(\lambda,r)$ is just $h(\lambda,r)=r^{\lambda}$. Hence $\mathbf{T}_f\left[r^{\lambda}\right]$ is a solution of (\ref{eigenvalueequation}).  Changing variables, the operator $\mathbf{T}_f$ can be written as
	\[
	\mathbf{T}_f\left[r^{\lambda}\right]= r^{\lambda}+\frac{1}{2}\int_0^1G^f(r,t)(1-t)^{\lambda}r^{\lambda}dt.
	\]
	Thus, $\frac{1}{r^{\lambda}}\mathbf{T}_f\left[r^{\lambda}\right]=1+\int_0^1G^f(r,t)(1-t)^{\lambda}dt\rightarrow 1$,  when $r\rightarrow 0^+$, by condition (\ref{initialconditionskernelG}). On the other hand
	\[
	\frac{1}{\lambda r^{\lambda-1}}\frac{d}{dr}\mathbf{T}_f\left[r^{\lambda}\right]= 1+\frac{1}{2}\int_0^1\left\{G_r^f(r,t)(1-t)^{\lambda}r+G^f(r,t)(1-t)^{\lambda} \right\}dt
	\]
	thar tends to $1$ when $r\rightarrow 0^+$. Hence $u(\lambda,r)=\mathbf{T}_f[r^{\lambda}]$ is the regular solution of (\ref{eigenvalueequation}) at $r=0$. In particular $f(r)=\mathbf{T}_f[1]$.
\end{remark}
We denote the set of all polynomial functions $p(r)=\sum_{n=0}^Na_nr^n$ in $[0,\varrho_{\Omega}]$ by $\mathcal{P}[0, \varrho_{\Omega}]$.
\begin{theorem}
	 For all $p\in \mathcal{P}[0,\varrho_{\Omega}]$,  the following equality is valid
	\begin{equation}\label{transmutation2firstform}
		\mathbf{T}_{\frac{1}{f}}p(r)= \frac{1}{f(r)}\left(\int_0^r\frac{f(s)\mathbf{T}_f\left[sp'(s)\right]}{s}ds+p(0)\right), \quad r\in [0,\varrho_{\Omega}].
	\end{equation}
\end{theorem}
\begin{proof}
	Let $n\in \mathbb{N}_0$ and $h_n(r)=r^n$. If $n=0$, $h_0\equiv 1$. Then the right hand side of (\ref{transmutation2firstform}) is $\frac{1}{f}$. On the other hand, by Remark (\ref{remarktransmutespowers}) $\mathbf{T}_{\frac{1}{f}}[1]$ is the regular solution of $\widehat{\mathbf{L}}_{\frac{1}{f}}y=0$ at $r=0$, that is given by $\frac{1}{f}$, and hence the equality is valid. For $n\geqslant 1$, denote\\ $g_n(r)=\left(\frac{1}{f(r)}\int_0^r\frac{f(s)\mathbf{T}_f\left[sh'_n(s)\right]}{s}ds+h_n(0)\right)$. We have
	\begin{align*}
		g_n(r) & = \frac{1}{f(r)}\int_0^r\frac{f(s)\mathbf{T}_f\left[sh'_n(s)\right]}{s}ds\\
		& = \frac{1}{f(r)}\int_0^r\frac{f(s)}{s}\mathbf{T}_f[ns^n]ds,
	\end{align*}
	and note that
	\[
	\frac{1}{s}\mathbf{T}_f[s^n]=s^{n-1}+\int_0^1G^f(r,t)(1-t)^ns^{n-1}ds,
	\]
	that is continuous in the interval $[0, \varrho_{\Omega}]$. Hence the integral in $g_n(r)$ is well defined in $[0,\varrho_{\Omega}]$. Thus,
	\[
	g_n(r)=\frac{n}{f(r)}\int_0^rf(s)s^{n-1}ds+\frac{n}{f(r)}\int_0^rf(s)\left[\int_0^1G^f(s,t)(1-t)^ns^{n-1}dt\right] ds.
	\] 
	By the L'Hopital rule we have
	\begin{align*}
		\lim_{r\rightarrow 0^+}\frac{n}{r^n}\int_0^rf(s)s^{n-1} & =\lim_{r\rightarrow 0^+}\frac{nf(r)r^{n-1}}{nr^{n-1}}= \lim_{r\rightarrow 0^+}f(r)=1.
	\end{align*}
 For the second integral, using the L'Hopital rule we have
\begin{align*}
	\lim_{r\rightarrow 0^+}\frac{1}{r^n}\int_0^rf(s)\left[\int_0^1G^f(s,t)(1-t)^ns^{n-1}dt\right]& = \lim_{r\rightarrow 0^+}\frac{1}{nr^{n-1}}f(r)\int_0^rG^f(r,t)(1-t)^nr^{n-1}dt\\
	& = \frac{1}{n}\lim_{r\rightarrow 0^+}f(r)\int_0^rG^f(r,t)(1-t)^ndt=0.
\end{align*}
Hence $g_n(r)\sim r^n$, $r\rightarrow 0^+$. For the derivative,
\[
g_n'(r)= -\frac{nf'(r)}{f(r)}\int_0^rf(s)\mathbf{T}_f[s^n]ds+\frac{n}{r}\mathbf{T}_f[r^n],
\] and then
\[
\frac{g'(r)}{nr^{n-1}}= -\frac{rf'(r)}{f^2(r)}\left(\frac{1}{r^{n}}\int_0^rf(s)\mathbf{T}_f[s^n]ds\right)+\frac{1}{r^n}\mathbf{T}_f[r^n].
\]
  We have just proven that $\lim_{r\rightarrow 0^+}\frac{n}{r^n}\int_0^rf(s)\mathbf{T}_f[s^n]=1$, then by (\ref{initialconditionsf}) and Remark \ref{remarktransmutespowers} we have that $g'(r)\sim nr^{n-1}$, $r\rightarrow 0^+$. Hence, $g_n(r)$ satisfies the asymptotic conditions (\ref{asymp2}). Finally, denote the operator of the right-hand side of (\ref{transmutation2firstform}) by $\widetilde{T}_2h_n$ and note that $\widehat{\mathbf{L}}_1[rh_n']= n^2rh'_n$ . Hence we have
\[
\widehat{\mathbf{L}}_{\frac{1}{f}}\widetilde{T}_2h_n= \mathbf{D}_{\frac{1}{f}}\mathbf{D}_f\widetilde{T}_2h_n=\mathbf{D}_{\frac{1}{f}}\mathbf{T}_f[rh_n']
\] and then 
\[
\mathbf{D}_{f} \widehat{\mathbf{L}}_{\frac{1}{f}}\widetilde{T}_2[h_n]=\widehat{\mathbf{L}}_f\mathbf{T}_f[rh_n']=\mathbf{T}_f\widehat{\mathbf{L}}_1[rh_n']=n^2\mathbf{T}_f[rh_n'].
\]
Since $\mathbf{D}_f\widetilde{T}_2h_n=\mathbf{T}_f[rh_n']$ we obtain that $\mathbf{D}_f\left(\widehat{\mathbf{L}}_{\frac{1}{f}}\widetilde{T}_2h_n\right)=n^2\mathbf{D}_f\widetilde{T}_fh_n$ in the interval $(0,\varrho_{\Omega}]$. Then $\widehat{\mathbf{L}}_{\frac{1}{f}}\widetilde{T}_2h_n=n^2\widetilde{T}_2h_n+\frac{c}{f}$ for some constant $c\in \mathbb{C}$. Note that $\widetilde{T}_2h_n(0)=0$ and 
\[
\mathbf{D}_{\frac{1}{f}}\mathbf{T}_f[rh_n']=rf\left(-\frac{f'}{f^2}\mathbf{T}_f[rh_n']+\frac{1}{f}\frac{d}{dr}\mathbf{T}_f[rh_n']\right).
\]
We have
\[
\frac{d}{dr}\mathbf{T}_f[r^n]=r^{n-1}+\int_0^1(1-t)^n\left\{G_r^f(r,t)r^n+nG(r,t)r^{n-1}\right\}.
\]
By the conditions (\ref{initialconditionskernelG}), $\frac{d}{dr}\mathbf{T}_f[r^n]=r^{n-1}\big{|}_{r=0}$. Thus, by (\ref{initialconditionsf}), $\mathbf{D}_{\frac{1}{f}}\mathbf{T}_f[rh_n']\big{|}_{r=0}=0$, which implies $\widehat{\mathbf{L}}_{\frac{1}{f}}\widetilde{T}_2h_n(0)=0$. Hence  $\widehat{\mathbf{L}}_{\frac{1}{f}}\widetilde{T}_2h_n=n^2\widetilde{T}_2h_n$. Since $\widetilde{T}h_n$ satisfies (\ref{asymp2}), it must be the regular solution at $r=0$. Thus, $\mathbf{T}_{\frac{1}{f}}h_n=\widetilde{T}_2h_n$. The equality is fulfilled for all $p\in \mathcal{P}[0,\varrho_{\Omega}]$ by the linearity of $\mathbf{T}_{\frac{1}{f}}$ and $\widetilde{T}_2$. 
\end{proof}

\begin{theorem}\label{theoremtransmutation2formulae}
	For all $u\in C^1[0,\varrho_{\Omega})$, the following equality holds
	\begin{equation}\label{transmutation2formulac1}
		\mathbf{T}_{\frac{1}{f}}u(r)=\frac{1}{f(r)}\left(\int_0^rf(s)\mathbf{T}_f[sh'(s)]ds+u(0) \right), \quad r\in [0,\varrho_{\Omega}).
	\end{equation}
\end{theorem}

\begin{proof}
	Let $u\in C^1[0, \varrho_{\Omega})$ and take $0<\rho<\varrho_{\Omega}$. Since $u\in C^1[0,\varrho]$, by the Weierstrass approximation theorem there exists a sequence $\{p_n\}_{n=0}^{\infty}\subset \mathcal{P}[0,\rho_{\Omega}]$ such that $p_n^{(k)}\overset{[0,\rho]}{\rightrightarrows} h^{(k)}$ when $n\rightarrow \infty$, $k=0,1$. Since $G_f, G_{\frac{1}{f}}\in C\left([0,\varrho_{\Omega}]\times [0,1]\right)$, we have that  $\mathbf{T}_{\frac{1}{f}}p_n\overset{[0,\rho]}{\rightrightarrows} \mathbf{T}_{\frac{1}{f}}u$ and $\mathbf{T}_f[sp'_n]\overset{[0,\rho]}{\rightrightarrows} \mathbf{T}_f[su']$, from where $\mathbf{T}_{\frac{1}{f}}u(r)=\widetilde{T}_2u(r)$ for all $r\in [0,\rho]$. Since $\rho$ was arbitrary, we obtain the equality in the whole segment $[0,\varrho_{\Omega})$.
\end{proof}

As a consequence of (\ref{theoremtransmutation2formulae}) we obtain the following transmutation relations.
\begin{proposition}
	The following equalities hold
	\begin{equation}\label{transmutationrelations2}
		\mathbf{D}_f\mathbf{T}_{\frac{1}{f}}u=\mathbf{T}\mathbf{D}_1u, \quad \mathbf{D}_{\frac{1}{f}}\mathbf{T}_fu=\mathbf{T}_{\frac{1}{f}}\mathbf{D}_1u, \quad \forall u\in C^1[0,\varrho_{\Omega}).
	\end{equation}
\end{proposition}

\begin{remark}
\begin{itemize}
	\item[(i)] Equality (\ref{transmutation2formulac1}) and the transmutation relations (\ref{transmutationrelations2}) are valid for $u\in C^1[0,\varrho_{\Omega}]$.
	\item[(ii)] If $u\in C^2[0,\varrho_{\Omega})$, then
	\[
	\widehat{\mathbf{L}}_f\mathbf{T}_fu=\mathbf{D}_f\mathbf{D}_{\frac{1}{f}}\mathbf{T}_fu=\mathbf{D}_f\mathbf{T}_{\frac{1}{f}}\mathbf{D}_1u= \mathbf{T}_f\mathbf{D}_1^2u=\mathbf{T}_f\widehat{\mathbf{L}}_1u.
	\] 
\end{itemize}		
\end{remark}

\subsection{Transmutation operators for the radial Vekua equation}
\begin{theorem}
	For all $W\in C^1(\Omega; \mathbb{B})$, the following equalities hold
\begin{eqnarray}
		r\left(\overline{\boldsymbol{\partial}}-\frac{\overline{\boldsymbol{\partial}}f}{f}C_{\mathbb{B}}\right)\mathcal{T}_fW & =&  \mathcal{T}_{\frac{1}{f}}r\overline{\boldsymbol{\partial}} W, \label{transmutationproertyvekua} \\
		 r\left(\boldsymbol{\partial}-\frac{\boldsymbol{\partial}f}{f}C_{\mathbb{B}}\right)\mathcal{T}_fW & = &  \mathcal{T}_{\frac{1}{f}}r\boldsymbol{\partial} W.
\end{eqnarray}
\end{theorem}
\begin{proof}
	We prove the first equality (the proof of the second is analogous). Let $W\in C^1(\Omega; \mathbb{B})$ and $u=\operatorname{Sc}W$, $v=\operatorname{Vec}W$. By (\ref{radialVekuaeq}) and (\ref{transmutationrelations2}) we have
	\begin{align*}
			r\left(\overline{\boldsymbol{\partial}}-\frac{\overline{\boldsymbol{\partial}}f}{f}C_{\mathbb{B}}\right)\mathcal{T}_fW & = \frac{e^{\mathbf{j}\theta}}{2}\left(r\frac{\partial}{\partial r}-\frac{rf'(r)}{f(r)}C_{\mathbb{B}}+\mathbf{j}\frac{\partial}{\partial \theta}\right)\left(\mathbf{T}_fu+\mathbf{j}\mathbf{T}_{\frac{1}{f}}v\right) \\
			& = \frac{e^{\mathbf{j}\theta}}{2}\left(\mathbf{D}_{\frac{1}{f}}\mathbf{T}_fu+\mathbf{j}\frac{\partial}{\partial \theta}\mathbf{T}_fu+\mathbf{j}\left(r\frac{\partial}{\partial r}\mathbf{T}_{\frac{1}{f}}v+\frac{rf'(r)}{f(r)}\mathbf{T}_{\frac{1}{f}}v\right)-\frac{\partial}{\partial \theta}\mathbf{T}_{\frac{1}{f}}v\right)\\
			&= \frac{e^{\mathbf{j}\theta}}{2}\left(\mathbf{T}_{\frac{1}{f}}\mathbf{D}_1u+\mathbf{j}\mathbf{T}_{f}u_{\theta}+\mathbf{j}\mathbf{D}_f\mathbf{T}_{\frac{1}{f}}v-\mathbf{T}_{\frac{1}{f}}v_{\theta}\right)\\
			&=\frac{e^{\mathbf{j}\theta}}{2}\left(\mathbf{T}_{\frac{1}{f}}\left(\mathbf{D}_1u-v_{\theta}\right)+\mathbf{j}\mathbf{T}_f\left(\mathbf{D}_1u+v_{\theta}\right)  \right)=\mathbf{T}_{\frac{1}{f}}\operatorname{Sc}r\overline{\boldsymbol{\partial}}W+\mathbf{j}\mathbf{T}_f\operatorname{Vec}r\overline{\boldsymbol{\partial}}W\\
			&= \mathcal{T}_{\frac{1}{f}}r\overline{\boldsymbol{\partial}}W.
	\end{align*}
\end{proof}
\begin{proposition}\label{proptransmuteholo}
	The following equality is valid 
	\begin{equation}
		\operatorname{V}_f(\Omega;\mathbb{B})=\mathcal{T}_f\left(\operatorname{Hol}(\Omega; \mathbb{B})\right).
	\end{equation}
\end{proposition}
\begin{proof}
	By the transmutation property (\ref{transmutationproertyvekua}), it is clear that $\mathcal{T}_fW\in \operatorname{V}_f(\Omega)$ if $W\in \operatorname{Hol}(\Omega;\mathbb{B})$. Reciprocally, if $W\in \operatorname{V}_f(\Omega)$, by Remark (\ref{remarktransmutationc2}) there exists $V\in C^2(\Omega; \mathbb{B})$ such that $W=\mathcal{T}_fV$. By (\ref{transmutationproertyvekua}) we obtain
	\[
	0=r\left(\overline{\boldsymbol{\partial}}-\frac{\overline{\boldsymbol{\partial}}f}{f}C_{\mathbb{B}}\right)\mathcal{T}_fV= \mathcal{T}_{\frac{1}{f}}r\overline{\boldsymbol{\partial}} V.
	\]
	Since $\mathcal{T}_{\frac{1}{f}}$ is a bijection, hence $\overline{\boldsymbol{\partial}}V=0$ in $\Omega$, that is, $V\in \operatorname{Hol}(\Omega; \mathbb{B})$.
\end{proof}

Denote the harmonic Bergman space by
\[
b_2(\Omega)=\{h\in \operatorname{Har}(\Omega)\, |\, h\in L_2(\Omega)\},
\]
and the Bergman space of solutions of $\mathbf{S}_fu=0$ as
\begin{equation}\label{bergmanradialschr}
	\operatorname{Sol}_2^{\mathbf{S}_f}(\Omega):= \{u\in \operatorname{Sol}_2^{\mathbf{S}_f}\, |\, u\in L_2(\Omega)\}.
\end{equation}
Since $\operatorname{Sol}^{\mathbf{S}_f}(\Omega)$ is closed in the Fr\'echet space $C(\Omega)$ \cite[Remark 13]{mine1}, then $\operatorname{Sol}_2^{\mathbf{S}_f}(\Omega)$ is a Hilbert space with reproducing kernel \cite[Prop. 2.3]{mine2}.
\begin{proposition}\label{propinvertibilityinvekua}
	The operator $\mathcal{T}_f:\mathcal{A}^2(\Omega; \mathbb{B})\rightarrow \mathcal{A}_f^2(\Omega; \mathbb{B})$ is bounded and invertible with bounded inverse. The same properties are valid for $\mathcal{T}_{\frac{1}{f}}$.
\end{proposition}
\begin{proof} Since $\mathbf{T}_f, \mathbf{T}_{\frac{1}{f}}\in \mathcal{B}\left(b_2(\Omega), L_2(\Omega)\right)$ with bounded inverses \cite[Sec. 3]{mine1}, then $\mathcal{T}_f\in \mathcal{B}\left(\mathcal{A}_2(\Omega; \mathbb{B}), L_2(\Omega; \mathbb{B})\right)$. If $W\in \mathcal{A}_f^2(\Omega; \mathbb{B})$, by Theorem \ref{theoremvekuaconjugate}, $\operatorname{Sc}W\in \operatorname{Sol}_2^{\mathbf{S}_f}(\Omega)$ and $\operatorname{Vec}W\in \operatorname{Sol}_2^{\mathbf{S}_{\frac{1}{f}}}(\Omega)$. Thus, $\mathbf{T}_f^{-1}\operatorname{Sc}W, \mathbf{T}_{\frac{1}{f}}^{-1}\operatorname{Vec}W\in b_2(\Omega)$ and then $\mathcal{T}_f^{-1}W\in L_2(\Omega; \mathbb{B})$. By Proposition \ref{proptransmuteholo}, $\mathcal{T}_f^{-1}W\in  \mathcal{A}_2(\Omega; \mathbb{B})$. Hence $\mathcal{T}_f: \mathcal{A}_2(\Omega; \mathbb{B})\rightarrow \mathcal{A}_2^f(\Omega; \mathbb{B})$ is a bounded bijection. By Theorem \ref{theorembergmanspacecomplete}, $\mathcal{A}_f^2(\Omega; \mathbb{B})$ is a complex Hilbert space, then it follows from the open mapping theorem \cite[Cor. 5.11]{folland} that  $\mathcal{T}_f^{-1}\in \mathcal{B}\left(\mathcal{A}_f^2(\Omega; \mathbb{B}), \mathcal{A}^2(\Omega; \mathbb{B})\right)$.
\end{proof}

\section{Complete system of solutions for the radial Vekua equation}
\subsection{The radial formal powers}
The following result allows us to know the action of the transmutation operator $\mathcal{T}_f$ over the set of polynomials in the variable $z$.
\begin{theorem}[\cite{mine1}, Th. 25]\label{theoremformalradialpowers1}
	For each $n\in \mathbb{N}_0$, the function $\mathbf{T}_f[z^n]$ is given by
	\begin{equation}
		 \mathbf{T}_f[r^ne^{ni\theta}]=\phi_f^{(n)}(r)r^ne^{in\theta},
	\end{equation}
where $\phi_f^{(n)}(r)=\frac{y_m^f(r)}{r^{m+\frac{1}{2}}}$ and $y_m^f(r)$ is the unique solution of the perturbed Bessel equation
\begin{equation}\label{perturbedbesselformalpowers1}
	-y_m''(r)+\frac{\left(m-\frac{1}{2}\right)\left(m+\frac{1}{2}\right)}{r^2}y_m(r)+q_f(r)y_m(r)=0, \quad 0<r\leqslant \varrho_{\Omega},
\end{equation}
satisfying the asymptotic conditions
\begin{equation}\label{asymptperturbedbesselformalpowers}
	y_m^f(r)\sim r^{m+\frac{1}{2}}, \quad \frac{d}{dr}y_m^f(r)\sim \left(m+\frac{1}{2}\right)r^{m-\frac{1}{2}}, \quad r\rightarrow 0^+
\end{equation}
\end{theorem}
Since $\mathbf{T}_f[r^ne^{in\theta}]=\mathbf{T}_f[r^n]e^{in\theta}$, we have the equality
\begin{equation}\label{transmutationpropertyradialone}
	\mathbf{T}_f[r^n]=\phi_f^{(n)}(r)r^n, \quad \forall n\in \mathbb{N}_0.
\end{equation}
\begin{remark}\label{remarkconstruction}
	A numerical method for the construction of the functions $\{\phi_f^{(n)}(r)\}_{n=0}^{\infty}$ based on the Spectral Parameter Power Series Method (SPPS) can be founded in \cite[Sec. 3]{castillo} and \cite[Sec. 6]{mine1}.
\end{remark}
\begin{definition}
	Let $n\in \mathbb{N}_0$. The {\bf basic Bicomplex radial formal powers of degree} $n$ (associated to $f$), are the functions given by
	\begin{eqnarray}
		\mathcal{Z}_f^{(n)}(1; z) & := & \mathcal{T}_f[\widehat{z}^n], \label{radialformalpower1} \\
		\mathcal{Z}_f^{(n)}(\mathbf{j};z) & := &  \mathcal{T}_f[\mathbf{j}\widehat{z}^n].\label{radialformalpower2}
	\end{eqnarray}
The family of all basic radial formal powers is given by $\{\mathcal{Z}_f^{(n)}(1;z), \mathcal{Z}_f^{(n)}(\mathbf{j},z)\}_{n=0}^{\infty}$.
\end{definition}
By Theorem (\ref{theoremformalradialpowers1}) and formulas (\ref{transmutationpropertyradialone}) and (\ref{relationmultiplicationbyj}), the basic Bicomplex radial formal powers can be written as
\begin{eqnarray}
	\mathcal{Z}_f^{(0)}(1; z) & = & f(r), \label{firstformalpower1}\\
	\mathcal{Z}_f^{(0)}(\mathbf{j}; z) & = & \frac{1}{f(r)}\mathbf{j} \label{firstformalpowerj}\\
	\mathcal{Z}_f^{(n)}(1; z) & = & r^n\left(\phi_f^{(n)}(r)\cos(n\theta)+\mathbf{j}\phi_{\frac{1}{f}}^{(n)}\sin(n\theta)\right), \quad n\geqslant 1, \label{bformalpowerpolar}\\
	\mathcal{Z}_f^{(n)}(\mathbf{j}; z) & = & r^n\left(-\phi_f^{(n)}(r)\sin(n\theta)+\mathbf{j}\phi_{\frac{1}{f}}^{(n)}(r)\cos(n\theta)\right)\quad n\geqslant 1. \label{bformalpowerpolarj}
\end{eqnarray} 
\begin{remark}\label{remakrpolynomialstoformalpowers}
	Let $P(z)=\sum_{n=0}^{M}A_n\widehat{z}^n$ a Bicomplex polynomial. By (\ref{noBlinearity}) we have
	\begin{align*}
		\mathcal{T}_fP(z) & =  \sum_{n=0}^{M}\mathcal{T}_f\left[A_n\widehat{z}^n\right] = \sum_{n=0}^{M}\left(\operatorname{Sc}(A_n)\mathcal{T}_f[\widehat{z}^n]+\operatorname{Vec}(A_n)\mathbf{j}\mathcal{T}_{\frac{1}{f}}[\widehat{z}^n]\right) \\
		& = \sum_{n=0}^{M}\left(\operatorname{Sc}(A_n)\mathcal{T}_f[\widehat{z}^n]+\operatorname{Vec}(A_n)\mathcal{T}_{f}[\mathbf{j}\widehat{z}^n]\right)\\
		&= \sum_{n=0}^{M}\left(\operatorname{Sc}(A_n)\mathcal{Z}_f^{(n)}(1; z)+\operatorname{Vec}(B_n)\mathcal{Z}_f^{(n)}(\mathbf{j}; z)\right),
\end{align*}
where in the third inequality we use (\ref{relationmultiplicationbyj}). 
\end{remark}

Motivated by this result and following \cite{CamposBicomplex}, we introduce the next definition.

\begin{definition}
	Let $n\in \mathbb{N}_0$ and $A\in \mathbb{B}$. The Bicomplex radial formal power of degree $n$ and coefficient $A$ is defined by
	\begin{equation}
		\mathcal{Z}_f^{(n)}(A; z):= \operatorname{Sc}(A)\mathcal{Z}_f^{(n)}(1;z)+\operatorname{Vec}(A)\mathcal{Z}_f^{(n)}(\mathbf{j};z).
	\end{equation}
A radial formal polynomial of degree $N\in \mathbb{N}_0$ is a sum of the form 
\begin{equation}
	S_N(z):= \sum_{n=0}^{N}\mathcal{Z}_f^{(n)}(A_n; z), \quad \mbox{ with }\; \{A_n\}_{n=0}^{N}\in \mathbb{B} \; \mbox{ and }\; A_N\neq 0.
\end{equation}
The set of all radial formal polynomials of degree $N$ is denoted by $\mathscr{S}_f^N(\Omega;\mathbb{B})$. We denote $\mathscr{S}_f(\Omega;\mathbb{B}):=\bigcup_{n=0}^{\infty}\mathscr{S}_f^N(\Omega;\mathbb{B})$.
\end{definition}

In particular, the basic Bicomplex formal powers correspond to the coefficients $1$ and $\mathbf{j}$. 
\begin{remark}\label{remarkpropertiesformalpolynomials}
Note that for $\alpha, \beta\in \mathbb{C}$ and $A\in \mathbb{B}$ we have
\begin{equation*}
	\alpha \mathcal{Z}_f^{(n)}(A;z)=  \mathcal{Z}_f^{(n)}(\alpha A; z), \quad \mathcal{Z}_f^{(n)}(\alpha+\mathbf{j}\beta; z)  = ´ \alpha\mathcal{Z}_f^{(n)}(1; z)+\beta\mathcal{Z}_f^{(n)}(\mathbf{j}; z).
\end{equation*}
Hence $\mathscr{S}_f^N(\Omega;\mathbb{B})=\operatorname{Span}_{\mathbb{C}}\{\mathcal{Z}_f^{(n)}(1;z),\mathcal{Z}_f^{(n)}(\mathbf{j};z) \}_{n=0}^{N}$ for $N\in \mathbb{N}_0$ and\\ $\mathscr{S}_f(\Omega;\mathbb{B})=\operatorname{Span}_{\mathbb{C}}\{\mathcal{Z}_f^{(n)}(1;z),\mathcal{Z}_f^{(n)}(\mathbf{j};z) \}_{n=0}^{\infty}$. 
By Remark \ref{remakrpolynomialstoformalpowers}, $\mathcal{T}_f\left[\operatorname{Span}_{\mathbb{C}}\{\widehat{z}^n, \mathbf{j}\widehat{z}^n\}_{n=0}^N\right]=\mathscr{S}_f^N(\Omega; \mathbb{B})$.   
\end{remark}
\subsection{Completeness in the space $\mathbf{V}_f(\Omega;\mathbb{B})$}
Given $X\subset \mathbb{C}$, the star-hull (with respect to $z=0$) of $X$ is the set $\operatorname{Star}(X):=\bigcup_{z\in X}[0,z]$, that is, the smallest star-shaped domain with respect to $z=0$ containing $X$. If $K\subset \mathbb{C}$ is compact, then $\operatorname{Star}(K)$ is also compact \cite[Lemma 7]{mine1}.

\begin{lemma}\label{lemmanormoperatort}
	For any compact $K\subset \Omega$ and $W\in C(\Omega;\mathbb{B})$, the following inequality holds
	\begin{equation}\label{innequalityoperator1}
		\max_{z\in K}|\mathcal{T}_fW(z)|_{\mathbb{B}}\leqslant M_1\max_{z\in \operatorname{Star}(K)}|W(z)|_{\mathbb{B}},	\end{equation}
where $M_1=2\max\left\{1+\frac{1}{2}\|G^f\|_{C([0,\varrho_{\Omega}]\times [0,1])},1+\frac{1}{2}\|G^{\frac{1}{f}}\|_{C([0,\varrho_{\Omega}]\times [0,1])} \right\}$.
\end{lemma}	
\begin{proof}
	Let $K\subset \Omega$ be compact. The operator $\mathbf{T}_f$ satisfies the following property \cite[Prop. 8]{mine1}
	\begin{equation*}
		\max_{z\in K}|\mathbf{T}_fu(z)|\leqslant \left(1+\frac{1}{2}\|G^f\|_{C([0,\varrho_{\Omega}]\times [0,1])}\right) \max_{z\in \operatorname{Star}(K)}|u(z)|, \quad \forall u\in C(\Omega).
	\end{equation*}
Thus, (\ref{innequalityoperator1}) follows from the fact that $|\mathcal{T}_fW(z)|_{\mathbb{B}}\leqslant |\mathbf{T}_f\operatorname{Sc}W(z)|+|\mathbf{T}_{\frac{1}{f}}\operatorname{Vec}W(z)|$.
\end{proof}
\begin{lemma}[\cite{mine2}, Lemma 5.10]\label{lemmasimplyconnected}
	If $X\subset  \mathbb{C}$ is a bounded set star-shaped with respect to $z=0$, then $\mathbb{C}\setminus X$ is connected. In particular, if $X=\Omega$  is a domain, then $\Omega$ is simply connected. 
\end{lemma}
\begin{lemma}[Runge's property]\label{lemmarunge}
	Let $V\in \operatorname{Hol}(\Omega;\mathbb{B})$ and $K\subset \Omega$ be compact. Given $\varepsilon>0$, there exists a Bicomplex polynomial $P_N(z)=\sum_{n=0}^{N}A_n \widehat{z}^n$, with $A_0, \cdots, A_N\in \mathbb{B}$, such that
	\begin{equation}\label{runge1}
		\max_{z\in K}|V(z)-P_N(z)|_{\mathbb{B}}<\varepsilon
	\end{equation}
\end{lemma}
\begin{proof}
	Since $V\in \operatorname{Hol}(\Omega; \mathbb{B})$, then $V^+$ is anti-holomorphic and $V^{-}$ is holomorphic. Hence, $\left(V^{-}\right)^{*}$ is holomorphic in $\Omega$. Since $\Omega$ is bounded and star-shaped with respect to $z=0$, by Lemma  \ref{lemmasimplyconnected} and the complex Runge's theorem \cite[Cor. 1.15]{conway}, there exists polynomials $p_1(z)=\sum_{n=0}^{N_1}a_nz^n$ and $p_2(z)=\sum_{n=0}^{N_2}b_nz^n$ such that
	\begin{equation*}
		\max_{z\in K}|\left(V^+(z)\right)^{*}-p_1(z)|<\frac{\varepsilon}{2\sqrt{2}}, \quad  \max_{z\in K}|V^{-}(z)-p_2(z)|<\frac{\varepsilon}{2\sqrt{2}}.
	\end{equation*}
Take $N=\max\{N_1,N_2\}$, define 
\[
\tilde{a}_n:=\begin{cases}
	a_n, & \mbox{ if } 0\leqslant n\leqslant N_1,\\
	0, & \mbox{ if } N_1<n\leqslant N,
\end{cases}, \quad 
\tilde{b}_n:=\begin{cases}
	b_n, & \mbox{ if } 0\leqslant n\leqslant N_1,\\
	0, & \mbox{ if } N_1<n\leqslant N,
\end{cases}
\]
and $\tilde{p_1}(z):=\sum_{n=0}^{N}\tilde{a}_nz^n$, $\tilde{p_2}(z):=\sum_{n=0}^{N}\tilde{b}_nz^n$. Finally, define $P_N(z)=\mathbf{p}^+(\tilde{p_1}(z))^*+\mathbf{p}^{-}\tilde{p_2}(z)$. By (\ref{powerz}), $P_N$ is a Bicomplex polynomial. Thus, for $z\in K$ we have
\begin{align*}
	|V(z)-P_N(z)|_{\mathbb{B}} &\leqslant \frac{1}{\sqrt{2}}\left(|V^+(z)-P^+(z)|+|V^{-}(z)-P^{-}(z)|\right)\\
	&= \frac{1}{\sqrt{2}}\left(|V^+(z)-\left(p_1(z)\right)^{*}|+|V^{-}(z)-p_1(z)|\right) \leqslant \frac{\varepsilon}{2},
\end{align*}
where in the first inequality we use (\ref{equivalentnorms}). Since $z\in K$ was arbitrary, we conclude that $\displaystyle \max_{z\in K}|V(z)-P_N(z)|<\varepsilon$.
\end{proof}

\begin{theorem}
	The Bicomplex radial formal powers are a complete system of solutions for the radial Vekua equation, that is, for any solution $W\in \operatorname{V}_f(\Omega; \mathbb{B})$ and any compact $K\subset \Omega$, there exists a sequence $\{S_n(z)\}\in\mathscr{S}_f(\Omega;\mathbb{B})$ such that $S_n \overset{K}{\rightrightarrows} W$. Furthermore, if $\Omega=B_{\varrho_{\Omega}}^{\mathbb{C}}(0)$, there exists constants $\{A_n\}_{n=0}^{\infty}\subset \mathbb{B}$ such that
	\begin{equation}\label{seriesofradialpowers}
		W(z)=\sum_{n=0}^{\infty}\mathcal{Z}_f^{(n)}(A_n;z), \quad z\in B_{\varrho_{\Omega}}^f(0),
	\end{equation}
and the series converges absolutely and uniformly on compact subsets of $B_{\varrho_{\Omega}}^{\mathbb{C}}(0)$.
\end{theorem}

\begin{proof}
	Let $K\subset \Omega$ be compact and $W\in \operatorname{V}_f(\Omega; \mathbb{B})$. By Proposition \ref{proptransmuteholo}, $V=\mathcal{T}_f^{-1}W\in \operatorname{Hol}(\Omega; \mathbb{B})$. For each $N\in \mathbb{N}$, since $\operatorname{Star}(K)$ is compact,  by Lemma \ref{lemmarunge} there exists a Bicomplex polynomial $P_N(z)=\sum_{n=0}^{M_N}B_n\hat{z}^n$ satisfying $\displaystyle \max_{z\in \operatorname{Star}(K)}|V(z)-P_N(z)|<\frac{1}{M_1N}$, where $M_1$ is defined as in Lemma \ref{lemmanormoperatort}. Set $S_N(z)= \mathcal{T}_fP_N(z)$. By Remark \ref{remarkpropertiesformalpolynomials}, $S_N\in \mathscr{S}_f(\Omega;\mathbb{B})$.  Given $z\in K$, by Lemma \ref{lemmanormoperatort} we have
	\begin{align*}
		|W(z)-S_N(z)|_{\mathbb{B}}=|\mathcal{T}_f(V-P_N)(z)|_{\mathbb{B}}\leqslant M_1\max_{z\in \operatorname{Star}(K)}|V(z)-P_N(z)|_{\mathbb{B}}\leqslant \frac{1}{N}.
	\end{align*}
Since $z\in K$ was arbitrary, we conclude that $\displaystyle \max_{z\in K}|W(z)-S_N(z)|_{\mathbb{B}}\leqslant \frac{1}{N}$. Thus, the sequence $\{S_N(z)\}_{n=0}^{\infty}$ converges uniformly to $W$ on $K$.

Now, suppose that $\Omega=B_{\varrho_{\Omega}}^{\mathbb{C}}(0)$. By Proposition \ref{propsitiontaylor}(ii), 
\begin{equation}\label{auxiliarseries1}
	V(z)=\sum_{n=0}A_n\widehat{z}^n, \quad \mbox{ with } A_n=\frac{\boldsymbol{\partial}^nV(0)}{n!}, n\in \mathbb{N},
\end{equation}
and the series converges in the topology of $C(\Omega; \mathbb{B})$. By the linearity and the continuity in $C(\Omega;\mathbb{B})$ (Proposition \ref{propertiesopt}) of $\mathcal{T}_f$ we have
\begin{align*}
	W(z)= \mathcal{T}_f\left[\sum_{n=0}^{\infty}A_n\widehat{z}^n\right]=\sum_{n=0}^{\infty}\mathcal{T}_f\left[A_n\widehat{z}^n\right]=\sum_{n=0}^{\infty}\mathcal{Z}_f^{(n)}(A_n;z)
\end{align*} 
and by the continuity of $\mathcal{T}_f$, the series converges in the topology of $C(\Omega; \mathbb{B})$. For the absolutely convergence, consider $0<r<\varrho_{\Omega}$ and $z\in \overline{B_r^{\mathbb{C}}(0)}$. By Lemma \ref{lemmanormoperatort} we obtain
\[
|\mathcal{T}_f[A_n\widehat{z}^n]|_{\mathbb{B}} \leqslant M_1 \max_{z\in \operatorname{Star}(\overline{B_r^{\mathbb{C}}(0)})}|A_n\widehat{z}^n|_{\mathbb{B}} \leqslant M_1\sqrt{2}|A_n|\max_{z\in \overline{B_r^{\mathbb{C}}(0)}}|\widehat{z}^n|=M_1\sqrt{2}|A_n|r^n
\]
(here, we use the fact that $\operatorname{Star}(X)=X$ if $X$ is star-shaped with respect to $z=0$). By Proposition \ref{propsitiontaylor}(ii), series (\ref{auxiliarseries1}) converges absolutely in $B_r^{\mathbb{C}}(0)$, hence $\sum_{n=0}^{\infty}|\mathcal{T}_f[A_n\widehat{z}^n]|_{\mathbb{C}}$ is dominated by the convergent series $\sum_{n=0}^{\infty}|A_n|r^n$. Therefore, (\ref{seriesofradialpowers}) converges absolutely in $\overline{B_r^{\mathbb{C}}(0)}$. 
\end{proof}
\subsection{An orthogonal basis for the Bergman space on a disk}
\begin{theorem}\label{theoremorthonormalbasisbergman}
	Suppose that $\Omega=B_{\varrho_{\Omega}}(0)$. The basic Bicomplex radial formal powers \\$\{\mathcal{Z}_f^{(n)}(1;z), \mathcal{Z}_f^{(n)}(\mathbf{j}; z)\}_{n=0}^{\infty}$ satisfy the following orthogonality relations for all $m,n\in \mathbb{N}$:
	\begin{align}
		\left\langle \mathcal{Z}_f^{(n)}(1; \cdot),\mathcal{Z}_f^{(m)}(\mathbf{j}; \cdot)\right\rangle_{L_2(\Omega; \mathbb{B})} & =0,\label{orthogonalityformalpowers1} \\
		 \left\langle \mathcal{Z}_f^{(n)}(\Lambda; \cdot),\mathcal{Z}_f^{(m)}(\Lambda; \cdot)\right\rangle_{L_2(\Omega; \mathbb{B})} & = \pi \left(\left\|\phi_f^{(n)}\right\|_{L_2(0,\varrho_{\Omega}; r^{n+1}dr)}^2+\left\|\phi_{\frac{1}{f}}^{(n)}\right\|_{L_2(0,\varrho_{\Omega}; r^{n+1}dr)}^2\right)\delta_{(n,m)}, \label{orthogonalityformalpowers2}\\
		 \left\langle \mathcal{Z}_f^{(0)}(1; \cdot),\mathcal{Z}_f^{(m)}(\Lambda; \cdot)\right\rangle_{L_2(\Omega; \mathbb{B})} & = \left\langle \mathcal{Z}_f^{(0)}(\mathbf{j}; \cdot),\mathcal{Z}_f^{(m)}(\Lambda; \cdot)\right\rangle_{L_2(\Omega; \mathbb{B})}=0 \label{orthogonalityformalpowers3}
	\end{align}
for $\Lambda \in \{1,\mathbf{j}\}$. Furthermore, the basic Bicomplex radial formal powers are an orthogonal basis for the Bergman space $\mathcal{A}_f^2(\Omega; \mathbb{B})$, 
\end{theorem}
\begin{proof}
	We prove (\ref{orthogonalityformalpowers1}) (the proof of (\ref{orthogonalityformalpowers2}) and (\ref{orthogonalityformalpowers3}) is analogous). Denote\\ $I_{n,m}= \left\langle \mathcal{Z}_f^{(n)}(1; \cdot),\mathcal{Z}_f^{(m)}(\mathbf{j}; \cdot)\right\rangle_{L_2(\Omega; \mathbb{B})}$.  We have
	\begin{align*}
		 I_{n,m}& = \iint\limits_{B_{\rho}^{\mathbb{C}}(0)}\Big{\{}\operatorname{Sc}\mathcal{Z}_f^{(n)}(1;z) \left(\operatorname{Sc}\mathcal{Z}_f^{(m)}(\mathbf{j};z)\right)^{*} +\operatorname{Vec}\mathcal{Z}_f^{(n)}(1;z)\left(\operatorname{Vec}\mathcal{Z}_f^{(m)}(\mathbf{j};z)\right)^{*} \Big{\}}dA_z\\
		&= \int_0^1r\int_0^{2\pi}\Bigg{\{}-r^{n+m}\phi_f^{(n)}(r)\left(\phi_{f}^{(m)}(r)\right)^{*}\cos(n\theta)\sin(m\theta)\\
		&\qquad\qquad+r^{n+m}\phi_{\frac{1}{f}}^{(n)}(r)\left(\phi_{\frac{1}{f}}^{(m)}(r)\right)^{*}\sin(n\theta)\cos(m\theta)\Bigg{\}}d\theta dr\\
		&= -\int_0^1r^{1+n+m}\phi_f^{(n)}(r)\left(\phi_{f}^{(m)}(r)\right)^{*}dr\int_0^{2\pi}\cos(n\theta)\sin(m\theta)d\theta \\
		& \qquad +\int_0^1r^{1+n+m}\phi_{\frac{1}{f}}^{(n)}(r)\left(\phi_{\frac{1}{f}}^{(m)}(r)\right)^{*}dr\int_0^{2\pi}\sin(n\theta)\cos(m\theta)d\theta =0.
	\end{align*}
Relations (\ref{orthogonalityformalpowers1})-(\ref{orthogonalityformalpowers3}) implies that $\{\mathcal{Z}_f^{(n)}(1; z), \mathcal{Z}_f^{(n)}(\mathbf{j};z)\}_{n=0}^{\infty}$ is an orthogonal system. Take $W\in \mathcal{A}_f^2(\Omega; \mathbb{B})$. By Proposition (\ref{propinvertibilityinvekua}), $V=\mathcal{T}_f^{-1}W\in \mathcal{A}^2(\Omega;\mathbb{B})$, and  by Proposition \ref{proporthonormalbasispowers} we can write
\[
V(z)= \sum_{n=0}^{\infty}\left[a_n \widehat{z}^n+b_n\mathbf{j}\widehat{z}^n\right] 
\]
for some coefficients $\{a_n,b_n\}_{n=0}^{\infty}\subset \mathbb{C}$. This series converges in $L_2(\Omega; \mathbb{B})$. Since\\$\mathcal{T}_f\in \mathcal{B}\left(\mathcal{A}^2(\Omega;\mathbb{B}), \mathcal{A}_f^2(\Omega; \mathbb{B})\right)$, we obtain 
\begin{align*}
	W(z) &=\mathcal{T}_fV(z)  = \sum_{n=0}^{\infty}\left(a_n \mathcal{T}_f[\widehat{z}^n]+b_n\mathcal{T}_f[\mathbf{j}\widehat{z}^n]\right)= \sum_{n=0}^{\infty}\left(a_n \mathcal{Z}_f^{(n)}(1,z)+b_n\mathcal{Z}_f^{(n)}(\mathbf{j},z)\right). 
\end{align*}
Hence, $W$ can be expanded into a Fourier series of Bicomplex radial formal powers. 
\end{proof}

When $\Omega$ is just bounded and star-shaped (with respect to $z=0$), the following result establishes some conditions for the completeness of formal powers.
\begin{theorem}
	If $\Omega= \operatorname{Int}(\overline{\Omega})$, then $\mathscr{S}_f(\Omega;\mathbb{B})$ is a complete system in $\mathcal{A}_f^2(\Omega)$.
\end{theorem}

\begin{proof}
	If $\Omega$ is star-shaped with respect to $z=0$, so is its closure. Indeed, given $z\in \overline{\Omega}$, take a sequence $\{z_n\}\subset \Omega$ such that $z_n\rightarrow z$. Then for any $t\in [0,1]$, $tz_n\in \Omega$ and $tz_n\rightarrow tz$. Hence $tz\in \overline{\Omega}$ for all $t\in [0,1]$. Since $\overline{\Omega}$ is bounded, by Lemma \ref{lemmasimplyconnected}, $\mathbb{C}\setminus \overline{\Omega}$ is a domain. This condition together with the hypothesis $\Omega= \operatorname{Int}(\overline{\Omega})$ and Lemma \ref{lemmasimplyconnected} implies that $\Omega$ is a Carath\'eodory domain \cite[Ch. 18, Prop. 1.9]{conway2}. Hence, $\{z^n\}_{n=0}^{\infty}$ is a complete system in the complex analytic Bergman space $\mathcal{A}^2(\Omega)$, and in consequence, $\{(z^{*})^n\}_{n=0}^{\infty}$ is complete in the anti-analytic Bergman space $\overline{\mathcal{A}}^2(\Omega)$ \cite[Ch. 18, Th. 1.11]{conway2}. Applying a procedure similar to that of the proof of Lemma \ref{lemmarunge} (changing the norm of the maximum by the $L_2$-norm) we obtain that $\{\widehat{z}^n, \mathbf{j}\widehat{z}^n\}_{n=0}^{\infty}$ is complete in $\mathcal{A}^2(\Omega; \mathbb{B})$. Let $W\in \mathcal{A}_f^2(\Omega;\mathbb{B})$. By Proposition \ref{propinvertibilityinvekua}, $V=\mathcal{T}_f^{-1}W\in \mathcal{A}^2(\Omega;\mathbb{B})$. Thus, given $\varepsilon>0 $, there exists a Bicomplex polynomial $P_N(z)= \sum_{n=0}^{N}A_n\widehat{z}^n$ such that $\|V-P_N\|_{L_2(\Omega;\mathbb{B})}<\frac{\varepsilon}{2M_2}$, where  $M_2=\|\mathcal{T}_f\|_{\mathcal{B}\left(\mathcal{A}^2(\Omega;\mathbb{B},\mathcal{A}_f^2(\Omega;\mathbb{B} )\right)}$. Take  $S_N(z)=\mathcal{T}_fP_N(z)\in \mathscr{S}_f^N(\Omega;\mathbb{B})$. Hence 
\begin{align*}
	\left\|W-S_N\right\|_{L_2(\Omega;\mathbb{B})}= \left\|\mathcal{T}_f\left(V-P_N\right)\right\|_{L_2(\Omega;\mathbb{B})}\leqslant M_2\left\|V-P_N\right\|_{L_2(\Omega;\mathbb{B})}\leqslant \frac{\varepsilon}{2}.
\end{align*}
$\therefore \mathscr{S}_f(\Omega;\mathbb{B})$ is complete in $\mathcal{A}_f(\Omega;\mathbb{B})$.
\end{proof}
\begin{proposition}
	Suppose that $\Omega=B_{\varrho_{\Omega}}^{\mathbb{C}}(0)$. Define the sequence of constants $\{M_0^1, M_0^2, M_n\}_{n=0}^{\infty}$ by
	\begin{eqnarray}
		M_0^1 & := & \sqrt{2\pi} \left\|f\right\|_{L_2(0,\varrho_{\Omega}; r^{n+1}dr)},\\
		M_0^2 & := & \sqrt{2\pi} \left\|\frac{1}{f}\right\|_{L_2(0,\varrho_{\Omega}; r^{n+1}dr)},\\
		M_n & := & \left(\pi \left(\left\|\phi_f^{(n)}\right\|_{L_2(0,\varrho_{\Omega}; r^{n+1}dr)}^2+\left\|\phi_{\frac{1}{f}}^{(n)}\right\|_{L_2(0,\varrho_{\Omega}; r^{n+1}dr)}^2\right)\right)^{\frac{1}{2}}, \quad n\geqslant 1.
	\end{eqnarray}
Then the Bergman kernel $\mathscr{K}_{\Omega}^f(A;z,\zeta)$ can be written as
\begin{align*}
	\mathscr{K}_{\Omega}^f(A;z,\zeta) & = \operatorname{Sc}(A)\frac{ f^{*}(\zeta)f(z)}{(M_0^1)^2}+\frac{\mathbf{j}\operatorname{Vec}(A)}{(M_0^1)^2 f^{*}(\zeta)f(z)} \\
	&\quad  + \sum_{n=1}^{\infty} \frac{1}{M_n^2}\left(\left\langle A, \mathcal{Z}_f^{(n)}(1; \zeta)\right\rangle_{\mathbb{B}} \mathcal{Z}_f^{(n)}(1; z)+\left\langle A, \mathcal{Z}_f^{(n)}(\mathbf{j}; \zeta)\right\rangle_{\mathbb{B}} \mathcal{Z}_f^{(n)}(\mathbf{j}; z)\right)
\end{align*}
The series converge with respect to $z$ in the $L_2(\Omega;\mathbb{B})$-norm, and uniformly on compact subsets of $\Omega$.
\end{proposition}
\begin{proof}
	Since $M_0^1= \left\|\mathcal{Z}_f^{(0)}(1;z)\right\|_{L_2(\Omega;\mathbb{B})}$, $M_0^2= \left\|\mathcal{Z}_f^{(0)}(\mathbf{j};z)\right\|_{L_2(\Omega;\mathbb{B})}$ and $M_n= \left\|\mathcal{Z}_f^{(n)}(\Lambda;z)\right\|_{L_2(\Omega;\mathbb{B})}$, $\Lambda \in \{1,\mathbf{j}\}$, $n\geqslant 1$ (by (\ref{firstformalpower1}), (\ref{firstformalpowerj}), and (\ref{orthogonalityformalpowers2})), the result follows from Remark \ref{remarkbergmankernelexpasion}, Theorem \ref{theoremorthonormalbasisbergman}, and the fact that
	\[
	\langle A, f(\zeta) \rangle_{\mathbb{B}}= \operatorname{Sc}(A)f^{*}(\zeta) \quad \mbox{and  }\; \left\langle A, \frac{\mathbf{j}}{f(\zeta)} \right\rangle_{\mathbb{B}}=\frac{\operatorname{Vec}(A)}{f^{*}(\zeta)}.
	\]
\end{proof}

\begin{example}
	Consider $f(r)= J_0(\kappa r)$, $q_f=-\kappa^2$, $q_{\frac{1}{f}}=3\kappa^2$, and $\Omega=\mathbb{D}$, as in Example \ref{examplehelmholtz}. By Theorem \ref{theoremformalradialpowers1}, $\phi_f^{(n)}= \frac{y_n^f}{r^{n+\frac{1}{2}}}$, $\phi_{\frac{1}{f}}^{(n)}= \frac{y_n^{\frac{1}{f}}}{r^{n+\frac{1}{2}}}$, where $y_n^{f}$ and $y_n^{\frac{1}{f}}$ satisfies the perturbed Bessel equation
	\begin{equation}\label{exampleperturbebbessel}
		-u_n''+\frac{\left(n+\frac{1}{2}\right)\left(n+\frac{1}{2}\right)}{r^2}u_n= \lambda u_n, \quad 0<r<1,
	\end{equation}
	with $\lambda= \kappa^2$ and $\lambda=-3\kappa^2$, respectively. According to \cite[Example 2.8]{castillo}, the regular solution $u_m(\lambda,r)$ of (\ref{exampleperturbebbessel}) that satisfies the asymptotic relations $u_n(\lambda,r)\sim r^{n+\frac{1}{2}}$, $u'_n(\lambda,r)\sim \left(n+\frac{1}{2}\right)r^{n-\frac{1}{2}}$, $r\rightarrow 0^{+}$, is given by
	\begin{equation}
		u_n(\lambda,r)= \Gamma\left(n+1\right)2^n\lambda^{-\frac{n}{2}}\sqrt{r}J_n(\sqrt{\lambda}r).
	\end{equation}
	Hence
	\begin{equation*}
		\phi_{\frac{1}{f}}^{(n)}= \frac{n!2^n}{\kappa^nr^n}J_n(\kappa r) \quad \phi_{\frac{1}{f}}^{(n)}(r)= \frac{n!2^n}{3^{\frac{n}{2}}\kappa^n}I_n(\sqrt{3}\kappa r),
	\end{equation*}
where $I_n$ stands for the modified Bessel function of the first kind. Thus, the basic radial formal powers are given by
\begin{align*}
	\mathcal{Z}_f^{(0)}(1; z) & = J_0(\kappa r), \quad \mathcal{Z}_f^{(0)}(\mathbf{j};z)= \frac{\mathbf{j}}{J_0(\kappa r)}, \\
	\mathcal{Z}_f^{(n)}(1;z) & = \frac{n!2^n}{\kappa^n}\left(J_n(\kappa r)\cos(n\theta)+\mathbf{j}3^{-\frac{n}{2}}I_n(\sqrt{3}\kappa r )\right)\\
	\mathcal{Z}_f^{(n)}(1;z)& = \frac{n!2^n}{\kappa^n}\mathbf{j}\left(3^{-\frac{n}{2}}I_n(\sqrt{3}\kappa r)\cos(n\theta)+\mathbf{j}J_n(\kappa r)\sin(n\theta)\right),
\end{align*}
and they are an orthogonal basis for the Bergman space associated to the Vekua equation $\overline{\boldsymbol{\partial}}W+\kappa \overline{W}=0$ in $\mathbb{D}$.
\end{example}

\begin{remark}
If $f$ is real-valued, (\ref{mainvekua}) is reduced to the complex main Vekua equation 
\begin{equation}\label{complexvekua}
	\frac{\partial}{\partial z^{*}}w(z)=\frac{f_{z^*}(z)}{f(z)}w^{*}(z),
\end{equation}
where $w$ is a complex-valued function. Since $q_f$ is real-valued, the kernel $G^f(r,t)$ is also real-valued \cite{gilbertatkinson}. Thus, operators $\mathbf{T}_f$ and $\mathbf{T}_{\frac{1}{f}}$ and the functions $\{\phi_n^f(r),  \phi_n^{\frac{1}{f}}(r)\}_{n=0}^{\infty}$ are real-valued. Hence the results presented are valid for the complex equation (\ref{complexvekua}), changing the scalar and vectorial parts for the real and imaginary parts, and the unit $\mathbf{j}$ for $i$. In this case, the Bergman space $\mathcal{A}_f^2(\Omega; \mathbb{C})$ is considered a real Hilbert space. With this approach, the results concerning completeness of $\mathcal{A}_f^2(\Omega;\mathbb{C})$ and the existence of the Bergman kernel coincides with those obtained in \cite{camposbergman}.
\end{remark}

\section{Conclusions}
A construction of a pair of transmutation operators that transmutes Bicomplex holomorphic functions into solutions of the radial main Vekua equation was presented. The construction was based on obtain a relationship between the transmutation operator of the associated radial Schr\"odinger equation and the transmutation operator of the corresponding Darbuoux transformed equation. The properties of continuity and invertibility in the space of classical solutions and in the pseudoanalytic Bergman space were studied. A complete system of solutions, called the radial formal powers, was obtained by transmuting the Bicomplex powers. The completeness of the radial formal powers was established in the sense of the uniform convergence on compact subsets and in the $L_2$-norm. The existence of a reproducing Bergman kernel for the Bicomplex pseudoanalytic Bergman space was proven. In the case of the Bergman space on a disk, the radial formal powers are an orthogonal basis and can be used to approximate the Bergman kernel.

\section*{Acknowledgments}
The author expresses his gratitude to Prof. Briceyda B. Delgado for helpful discussions.

\end{document}